\documentclass[aos,preprint]{imsart}
\usepackage[margin=4.1cm]{geometry}
\RequirePackage[OT1]{fontenc}
\usepackage[colorlinks,citecolor=blue,urlcolor=blue]{hyperref}

\usepackage{verbatim}
\usepackage{url}

\usepackage{graphicx}
\usepackage{epstopdf}
\usepackage{amsmath,amsfonts,amssymb,amsthm}
\usepackage{array}
\usepackage{graphicx}
\usepackage{bm}
\usepackage{multirow}
\usepackage[font=footnotesize]{caption}
\usepackage{subcaption}
\usepackage{algorithm}
\usepackage{algpseudocode}
\usepackage{cleveref}
\usepackage[shortlabels]{enumitem}

\crefname{equation}{}{}

\newfloat{algorithm}{t}{lop}

\usepackage{arash_macros}
\RequirePackage{xspace}
\RequirePackage{xparse}

\DeclareMathOperator{\lcc}{LC}

\newcommand\pimin{\pi_{\min}}
\newcommand\gams{\gamma_*}
\newcommand\gamst{\gamt_*}
\newcommand\vbs{\vb_*}

\DeclareMathOperator{\kron}{\otimes}                       
\DeclareMathOperator{\cov}{cov}                       
\newcommand\curvature{\alpha}
\newcommand{\lipnorm}[1]{\norm{#1}_{\text{Lip}}}
\newcommand{\pib}{\bar \pi}
\newcommand\rt{\widetilde{r}}
\newcommand\fb{\bar f}
\newcommand\convd{\rightsquigarrow}
\newcommand\ut{\widetilde{u}}
\newcommand\vt{\widetilde{v}}

\newcommand{\Uh}{\hat{U}}
\newcommand{\Lamh}{\hat{\Lambda}}
\newcommand{\lamh}{\hat{\lambda}}
\newcommand{\kmeans}{$k$-means\xspace}
\newcommand\Sigb{\bm{\Sigma}}

\DeclareMathSymbol{\R}{\mathbin}{AMSb}{"52}

\newcommand{\siginf}{\sigma_\infty}

\newcommand{\Xt}{\tilde{X}}

\newcommand{\Nnet}{\mathcal{N}}
\newcommand{\Kt}{\widetilde{K}}
\newcommand{\Ss}{{\mathsf S}}

\DeclareMathOperator{\Misb}{\overline{Mis}}

\newcommand\Ks{K^*}
\newcommand\Ph{\widehat{P}}
\newcommand\Cc{\mathcal C}
\newcommand\vb{\protect\widebar v}
\newcommand\gamt{\widetilde{\gamma}}
\newcommand\numc{R}
\newcommand\zh{\widehat z}
\newcommand\Yc{\mathcal Y}

\begin{document}

\begin{frontmatter}

\title{Concentration of kernel matrices with application to kernel spectral clustering}
\runtitle{Concentration of kernel matrices}

\begin{aug}
\author{Arash A. Amini\thanksref{ucla}\ead[label=earash]{aaamini@ucla.edu}}
\and
\author{Zahra S. Razaee\thanksref{cshs}\ead[label=ezahra]{zahra.razaee@cshs.org}}

\runauthor{Amini \and Razaee}

\affiliation{\thanksmark{ucla}\mbox{University of California, Los Angeles} and \thanksmark{cshs}\mbox{Cedars-Sinai Medical Center}}

\address{Department of Statistics\\
8125 Math Sciences Bldg.\\
Los Angeles, CA 90095-1554, USA\\
\printead{earash}}

\address{Biostatistics and Bioinformatics Research Center\\
700 N. San Vicente Blvd., G500\\
West Hollywood, CA 90069, USA\\
\printead{ezahra}}
\end{aug}

\begin{abstract}
	We study the concentration of random kernel matrices around their mean. We derive nonasymptotic exponential concentration inequalities for Lipschitz kernels assuming that the data points are independent draws from a class of multivariate distributions on $\reals^d$, including the strongly log-concave distributions under affine transformations. A feature of our result is that the data points need not have identical distributions or zero mean, which is key in certain applications such as clustering. Our bound for the Lipschitz kernels is dimension-free and sharp up to constants. For comparison, we also derive the companion result for the Euclidean (inner product) kernel for a class of sub-Gaussian distributions. A notable difference between the two cases is that, in contrast to the Euclidean kernel, in the Lipschitz case, the concentration inequality does not depend on the mean of the underlying vectors. As an application of these inequalities, we derive a bound on the misclassification rate of a kernel spectral clustering (KSC) algorithm, under a perturbed nonparametric mixture model. We show an example where this bound establishes the high-dimensional consistency (as $d \to \infty$) of the KSC, when applied with a Gaussian kernel, to a noisy model of nested nonlinear manifolds.
\end{abstract}


\end{frontmatter}

\section{Introduction}

Kernel methods are quite widespread in statistics and machine learning, since %
many  ``linear'' methods %
can be turned into nonlinear ones by replacing the Gram matrix with one based on a nonlinear kernel, the so-called kernel trick. The approach is often motivated as follows: One first maps the data $x \in \reals^d$ to a point $\Phi(x)$ in a higher dimensional space $H$ via a nonlinear feature map $\Phi: \reals^d \to H$. In this new space, the data are better behaved (e.g., linearly separated in the case of classification), hence one can run a simple linear algorithm. Often this algorithm relies only on the inner products $\ip{\Phi(x),\Phi(y)} = K(x,y)$. Thus the transformation is effectively equivalent to  replacing the usual inner product $\ip{x,y}$ with the kernelized version $K(x,y)$, keeping the computational cost of the algorithm roughly the same. 
This way of introducing  nonlinearity  without sacrificing efficiency, works well for many commonly used algorithms such as principal component analysis, ridge regression, support vector machines, $k$-means clustering, and so on~\cite{shawe2005eigenspectrum,blanchard2007statistical,yang2017randomized}.

To be concrete, let the data be the random sample $X_1,\dots,X_n \in \reals^d$  drawn independently from  unknown distributions $P_1,\dots,P_n$. Then, the kernel trick replaces the Gram matrix $(\ip{X_i,X_j}) \in \reals^{n \times n}$ with the random \emph{kernel matrix} $K(X) := \big(K(X_i,X_j) \big) \in \reals^{n \times n}$. %
Understanding the behavior of this random matrix, and especially how well it concentrates around its mean is key in evaluating the performance of the underlying kernel methods. This problem has been studied in the literature, but often in the asymptotic setting, including the classical asymptotics where $d$ is fixed and $n \to \infty$ or in the (moderately) high-dimensional regime where $d,n \to\infty$ and $d/n \to \gamma \in (0,1)$.

In this paper, we study finite-sample concentration of $K(X)$ around its mean in the $\ell_2$ operator norm, i.e., $\norm{K(X) - \ex K(X)}$.  We will make no assumptions about the relative sizes of $d$ and $n$; our results hold for any scalings of the pair $(n,d)$. We also do not assume the kernel (function) to be positive semidefinite, using the term kernel broadly to refer to any symmetric real-valued function defined on $\reals^d \times \reals^d$. We consider the class of Lipschitz kernels and provide a concentration inequality when the data distributions $\{P_i\}$ correspond to certain classes of distributions, including the strongly log-concave distributions in $\reals^d$. In particular, the result holds for general Gaussian distributions  $P_i = N(\mu_i, \Sigma_i), i=1,\dots,n$. 
For comparison, we also derive a concentration inequality for the usual Euclidean kernel, for certain classes of sub-Gaussian vectors.
 Our results highlight differences in dimension dependence between the concentration of Lipschitz kernels versus that of the Euclidean one. Another interesting observation is that, in contrast to the Euclidean case, the concentration inequality for Lipschitz kernels does not depend on the mean kernel $\ex K(X)$.

A feature of our results is that the data, although independent, are not assumed to be identically distributed. This is important, for example, when studying clustering problems and implies that 
the mean kernel matrix $\ex K(X)$ is nontrivial and can carry information about the underlying data distribution. Thus, one can study the behavior of a kernel method on the mean matrix $\ex K(X)$ and then translate the results to a random sample, using the concentration equality. %

We illustrate this approach by analyzing a kernel spectral clustering algorithm which is recently introduced in the context of network clustering. We adapt the algorithm to general kernel clustering, and provide bounds on its misclassification rate under  a (nonparametric) mixture model that is perturbed by noise. Due to our concentration results, the bound we derive allows for anisotropic noise models as well as noise structures that vary with the signal. 
This, in turn, allows one to investigate an interesting trade-off between the noise and signal structure. There could be multiple ways of breaking the data into the signal and noise components.  For example, consider $X_i  = \mu_i + \eps_i$ where $\mu_i$ is the signal component and $\eps_i \sim N(0, \Sigma)$ the independent isotropic noise. An alternative decomposition is 
\[ X_i = \mu'_i+ \eps'_i \quad \text{for}\quad \mu_i' = \mu_i + \Pi_{\mu_i}^\perp \eps_i,\quad \eps'_i = \Pi_{\mu_i}\eps_i\]
where  $\Pi_{\mu_i}$  is the operator projecting onto the span of $\{\mu_i\}$, and $\Pi_{\mu_i}^\perp = I_d - \Pi_{\mu_i}$ is its complementary projection operator. This latter decomposition has varying anisotropic noise $\eps'_i \sim N(0,  \Pi_{\mu_i} \Sigma \Pi_{\mu_i})$, but could allow for faster concentration of the kernel matrix (conditioned on $\{\mu_i\})$ when $\max_i \opnorm{ \Pi_{\mu_i} \Sigma \Pi_{\mu_i}}$ is smaller than $\opnorm{\Sigma}$.
We illustrate the application of our concentration bound by analyzing a nested sphere cluster model under isotropic and radial noise models, and show that the proposed kernel spectral clustering algorithm achieves high-dimensional consistency under both noise structures.

 In addition to the trade-off in decomposition, the bound on the misclassification rate also shows an interesting trade-off between the approximation (by a block-constant matrix) and estimation errors. This trade-off is controlled by  certain parameters of the mean kernel $\ex K(X)$, denoted as $\gamma^2$ and $\vb^2$ in Section~\ref{sec:cluster}, that characterize the between-cluster distance and the 
 within-cluster variation. Both of these are further affected by the noise level $\sigma$ and, in the case of the Gaussian kernel, by the kernel bandwidth.
\subsection{Related work}

Most of the prior work  on the concentration of kernel matrices  focuses on the asymptotic behavior. For fixed $d$, as $n \to \infty$, the eigenvalues of the normalized kernel matrix $K(X)/n$ converge to the eigenvalues of the associated integral operator if (and only if) the operator is Hilbert-Schmidt. This is shown in~\cite{koltchinskii2000random} which also provides rates of convergence and distributional limits.

More recently, the so-called high-dimensional asymptotic regime where $n,d \to \infty$ while $d/n$ converges to a constant is considered. The study of kernel matrices in this regime was initiated by~\cite{el2010spectrum} where it was shown that for kernels with entries of the form $f(X_i^T X_j)$ and  $f(\norm{X_i-X_j})$, under a certain scaling of the distribution of $\{X_i\}$, the empirical kernel matrix asymptotically behaves similar to that obtained from a linear (i.e., Euclidean) kernel. 

In particular, it was shown in~\cite{el2010spectrum} that the operator norm distance between the kernel matrix and its linearized version vanishes asymptotically, hence for example, the corresponding spectral densities approach each other. The limiting spectral density (i.e., the limit of the empirical density of the eigenvalues) has been further studied for kernels with entries of the form $f(X_i^T X_j)$ and  $f(\norm{X_i-X_j})$ in~\cite{cheng2013spectrum,do2013spectrum,fan2015spectral} under various (often relaxed) regularity assumptions on $f$ and the distribution of $\{X_i\}$.
In parallel work, \cite{el2010information} considers a signal-plus-noise model for $X_i$ and shows that the kernel matrix, in this case, approaches a kernel matrix which is based on the signal component alone. Although, the results are mostly asymptotic, they have similarities with our approach. We make a detailed comparison with~\cite{el2010information} in Remarks~\ref{rem:compare:el2010info} and Section~\ref{sec:comparison}. 

 Early results on finite-sample concentration bounds for kernel matrices include~\cite{shawe2002concentration,blanchard2007statistical,braun2006accurate} for individual   eigenvalues or their partial sums. In~\cite{braun2006accurate,blanchard2007statistical},  the deviation of the eigenvalues of the empirical kernel matrices (or their partial sums) from their counterparts based on the associated integral operator are considered. In~\cite{shawe2002concentration}, non-asymptotic concentration bounds on the eigenvalues have been obtained for bounded kernels. In our notation, these bounds show that $|\lambda_i(K) - \ex \lambda_i(K)|$ are small. In contrast, a consequence of our results is a control on $|\lambda_i(K) - \lambda_i(\ex K)|$. In applications, getting a handle on  $\lambda_i(\ex K)$ is often much easier than $\ex \lambda_i(K)$.

More recently,   sharp non-asymptotic upper bounds on the operator norm of random kernel matrices were obtained in~\cite{kasiviswanathan2015spectral} for the case of polynomial and Gaussian kernels. These results focus on the case where $X_i$ are \emph{centered} sub-Gaussian vectors and provide direct bounds on the operator norm of the kernel matrix: $\norm{K}$. In contrast, we focus on the case where $X_i$ have a non-zero mean $\mu_i$ and $\ex K$ has nontrivial information about these mean vectors, and we provide bounds on the deviation of $K$ from $\ex K$.

Much of the work on the analysis of spectral clustering focuses on the Laplacian-based approach. In a line of work, the convergence of the \emph{adaptive} graph Laplacian to the corresponding Laplace-Beltrami operator is established~\cite{belkin2003laplacian,hein2005graphs,hein2006uniform,singer2006graph,gine2006empirical}. For a \emph{fixed} kernel, the convergence of the (empirical) graph Laplacian to the corresponding population-level integral operator is studied in~\cite{von2008consistency,rosasco2010learning}, and bounds on the deviation of the corresponding spectral projection operators are derived. More recently, a finite-sample analysis for fixed kernels is provided in~\cite{schiebinger2015geometry} assuming  an explicit mixture model for the data. 
Our work is close in spirit to~\cite{schiebinger2015geometry} with notable differences. We consider an adjacency-based kernel  spectral clustering, based on a recently proposed algorithm for network clustering, and provide direct bounds on its misclassification rate.  Our bound requires no assumption on the signal structure, and the overall bound is simpler and in terms of explicit quantities related to the statistical properties of a mean kernel. We separate the contributions of the noise and signal  (in contrast to~\cite{schiebinger2015geometry}), which allows for a more refined analysis. In particular, we show how this could lead to high-dimensional consistency of the proposed kernel spectral clustering in some examples. Another recent work in the same spirit as ours is that of~\cite{yan2016robustness} where both a spectral method and a SDP relaxation are analyzed for  clustering based on a kernel matrix. A mixture model with isotropic sub-Gaussian noise is considered in~\cite{yan2016robustness} and consistency results are obtained for both approaches, based on entrywise concentration bounds for the kernel matrix. We provide more detailed comparisons with the existing literature on kernel clustering in Section~\ref{sec:comparison}.

\medskip
The rest of the paper is organized as follows: In Section~\ref{sec:concent:ineq}, we derive the concentration inequalities for the Lipschitz and Euclidean kernels. Section~\ref{sec:cluster} presents an application of these results in deriving misclassification  bounds for kernel spectral clustering. In Section~\ref{sec:sims}, we present simulation results corroborating the theory. We conclude by giving the proofs of the  main results in Section~\ref{sec:proofs}, leaving some details to the appendices in the~\ref{supp}.
\section{Concentration of kernel matrices}\label{sec:concent:ineq}

Throughout, $\{X_i, i=1,\dots, n\}$ will be a collection of independent random vectors in $\reals^d$. The sequence is not assumed i.i.d., that is, the distribution of $X_i$ could in general depend on $i$. This for example is relevant to clustering applications. We will collect $\{X_i\}$ into the data matrix $X = (X_1,\dots,X_n) \in \reals^{ d\times n}$. We also use the notation $X = (X_1 \mid \cdots \mid X_n)$ to emphasize that $X_i$ is the $i$th column of $X$. For a vector $x \in \reals^n$, $\norm{x} = \norm{x}_2$ denotes the $\ell_2$ norm. For a matrix $A \in \reals^{n \times n}$, we use $\opnorm{A}$ to denote the $\ell_2$ operator norm, also known as the spectral norm.

We are interested in bounds on the deviation $\opnorm{K - \ex K}$, where $K = (K_{ij}) \in \reals^{n \times n}$ is a kernel matrix. That is, $K_{ij} = K(X_i,X_j)$, where with some abuse of notation, we will use the same symbol $K$ to denote both the kernel matrix and the kernel function $K : \reals^d \times \reals^d \to \reals$. Occasionally, we write $K(X)$ for the kernel matrix when we want to emphasize the dependence on $X$. Thus,
\begin{align}\label{eq:K(X)}
	K(X) = \big( K(X_i,X_j) \big) \in \reals^{n \times n}.
\end{align} 
For a random vector $X_i$, we denote its covariance matrix as $\cov(X_i)$. We often work with Lipschitz functions. A function $f :  \reals^d \to \reals$ is Lipschitz with respect to (w.r.t.) metric $\delta$ on $\reals^d$ if it has a finite Lipschitz semi-norm:
\begin{align*}
	\lipnorm{f} := \sup_{x,y} \frac{|f(x) - f(y)|}{\delta(x,y)} < \infty.
\end{align*}
It is called $L$-Lipschitz if $\lipnorm{f} \le L$. If the metric is not specified, it is assumed to be the Euclidean metric, %
$\delta(x,y) := \norm{x-y}$.

We  consider the data model $X_i = \mu_i + \sqrt{\Sigma_i} W_i, i = 1,\dots,n$, where $\Sigma_i$ is a \emph{generalized square-root} of the positive semidefinite matrix $\Sigma_i$, in the sense that $\sqrt{\Sigma_i} \sqrt{\Sigma_i}^T = \Sigma_i$. Note that $\sqrt{\Sigma_i}$ need not be symmetric. %

\subsection{Lipschitz kernels}
Our first result is for the case where the kernel function $K : \reals^d \times \reals^d \to \reals$ is $L$-Lipschitz, in the following sense:
\begin{align}\label{eq:Lip:kern:ineq}
|K(x_1,x_2) - K(y_1,y_2)| \le L ( \norm{x_1 - y_1} + \norm{x_2 - y_2}).
\end{align}
This class includes any kernel function which is $L$-Lipschitz w.r.t. the $\ell_2$ norm on $\reals^{2d}$. It also includes the important class of \emph{distance kernels} of the form (see Appendix~\ref{app:dist:kern:lip} in the~\ref{supp}): %
\begin{align}\label{eq:dist:kern}
	 K(x_1,x_2) = f(\norm{x_1 - x_2}), \quad \text{$f:\reals \to \reals$ is $L$-Lipschitz},
\end{align}
which in turn includes the important case of the Gaussian kernel where $f(t) \propto e^{-t^2/2\sigma^2}$. We also need the following definition:
\begin{defn}
	We say that a random vector $Z \in \reals^d$ is strongly log-concave with curvature $\alpha^2$ if it has a density $f(x) = e^{-U(x)}$  (w.r.t. the Lebesgue measure) such that $\nabla^2 U (x) \succeq \alpha^2 I_d $ for  every $x \in \reals^d$, i.e., the Hessian of $U$ exists and is uniformly bounded below.
\end{defn}
We often work with the following class of multivariate distributions:
\begin{defn}[$\lcc$ class]\label{defn:LC}
	We say that random vector $X \in \reals^d$ belongs to class $\lcc(\mu,\Sigma, \omega)$ for some vector $\mu \in \reals^d$,  a $d\times d$ semidefinite matrix $\Sigma$ and $\omega > 0$, if we can write $X = \mu + \sqrt{\Sigma}\, W$ where $W \in \reals^d$ is a random vector whose $j$th coordinate, $W_{j}$, satisfies $\ex W_{j} = 0$ and $\ex W_{j}^2=1$ for all $j$. Moreover, either of the following conditions hold:
	\begin{enumerate}[(a)]
		\item $W_{j} = \phi_{j}(Z_{j})$, for some function $\phi_j$ with $\lipnorm{\phi_{j}} \le \omega$, for all $j$, and $\{Z_{j}\}$ is a collection of independent standard normal variables; or
		\item  $\{W_{j}\}$ are independent and $W_{j}$ has a density (w.r.t. the Lebesgue measure) uniformly bounded below by $1/\omega$; or
		\item $W$ is strongly log-concave with curvature $\alpha^2 \ge 1/\omega^2$, and $\ex W W^T = I_d$.
	\end{enumerate}
\end{defn}

For part~(b) of Definition~\ref{defn:LC}, we say that a density $f$ is uniformly bounded below, if $f(x) \ge 1/\omega >0$ for all $x$ in the support of the  distribution. Part~(b) thus includes the case where the marginals of $X$ are uniformly distributed on bounded subsets of $\reals$ and $\cov(X)^{-1/2}(X - \ex X)$ has independent coordinates. Note that a multivariate Gaussian random vector is a special case of Definition~\ref{defn:LC} with $\omega = 1$.
Our main result for the Lipschitz kernels is the following:

\begin{thm}\label{thm:Lip:ker:gen}
	Let $X_i \in \lcc(\mu_i,\Sigma_i, \omega)$, $i=1,\dots,n$, be a collection of independent random vectors, and
	let $K = K(X)$ be the kernel matrix in~\eqref{eq:K(X)} with kernel function satisfying~\eqref{eq:Lip:kern:ineq}. Then, for some universal constant $c > 0$, with probability at least $1-\exp(-c \,t^2)$,
	\begin{align}\label{eq:dist:ker:Gauss}
	\opnorm{K - \ex K} \; \le \;2 L \omega \siginf (C n + \sqrt{n} t)
	\end{align}
	where $\siginf^2 := \max_i \opnorm{\Sigma_i}$ and $C = c^{-1/2}$.
	When all $X_i$s are multivariate Gaussians, one can take $c = 1/2$.
\end{thm}
Although this result is stated for the LC classes of random vectors, it holds more broadly. In fact, we can even relax the independence assumption on $W_1,\dots,W_n$. Inspection of the proof shows that the result holds as long as $\vec W \in \reals^{d n}$, which is obtained by stacking $\{W_i\}$ on top of each other, satisfies the so-called \emph{concentration property}; see Definition~\ref{defn:concent:property} in Section~\ref{sec:proofs}.

Bound~\eqref{eq:dist:ker:Gauss} is dimension-free. To see this, consider the case 
where $\Sigma_i = \sigma^2 I_d$ for all~$i$. Then, %
we have $\frac1n 	\opnorm{K - \ex K} = O( L \omega \sigma )$ with  probability at least $1-e^{- cnt^2}$, for all $d$. The bound is also independent of $\{\mu_i\}$.
The following proposition shows the bound %
is sharp:

\begin{prop}\label{prop:lip:lower:bound}
    Let $X_i,i=1,\dots,n$ be i.i.d. draws from a symmetric distribution with
     $\pr(|X_i| > \sigma) = 1/2$, e.g., the uniform distribution on $(-2\sigma,2\sigma)$. Then, for any $\sigma > 0$, there is an $L$-Lipschitz kernel function on $\reals$ such that, when $n \ge 8$, the corresponding kernel matrix $K = K(X)$ satisfies
    \begin{align}\label{eq:lip:lower:bound}
        \pr\big( 
        \opnorm{K - \ex K} >  L \sigma n / 8 \big) \ge 1-e^{-n/8}.
    \end{align}
\end{prop}
The $1/2$ in assumption $\pr(|X_i| > \sigma) = 1/2$, is for convenience. It can be replaced with any positive constant by modifying the constants in~\eqref{eq:lip:lower:bound}.

\begin{rem}\label{rem:compare:el2010info} As an intermediate step in proving Theorem~\ref{thm:Lip:ker:gen}, we obtain %
	(cf. Proposition~\ref{prop:E:frob:dev:Lip}), %
	\begin{align}\label{eq:ex:fro:dev}
	\frac1{n^2} \ex \fnorm{K - \ex K}^2 \le \frac{4}{c} L^2 \omega^2 \max_i \opnorm{\Sigma_i}.  %
	\end{align}
	This is a significant strengthening of a result that 
	follows from Theorem~1 in~\cite{el2010information}: After a rescaling to match the two models, the result there implies
	\begin{align}\label{eq:elkaroui}
	\frac1{n^2}\ex \norm{K - \Kt}_F^2 \le C L^2 \big[\tr(\Sigma^2) + C_1 \norm{\Sigma}\, \big]
	\end{align}
	for the case where $\Sigma_i = \Sigma$ for all $i$, the kernel is of the form~\eqref{eq:dist:kern} and $\Kt$ is a modified kernel matrix where $f(\cdot)$ %
	is replaced with $f(\cdot + \tr(\Sigma))$ off the diagonal and with $f(0)$ on the diagonal.
	
	Our result is much sharper since the bound does not scale with $d$. It is also more general in some aspects, namely, that it applies to any Lipschitz kernel, not necessarily of the form~\eqref{eq:dist:kern}, and we allow for heterogeneity in the covariance matrices of the data points. Our result is stated in terms of the mean matrix $\ex K$ which is a more natural object. 
	Moreover, we prove a full concentration result in Theorem~\ref{thm:Lip:ker:gen} which goes beyond controlling the mean of the deviation as in~\eqref{eq:ex:fro:dev} and~\eqref{eq:elkaroui}. On the other hand, the result in~\cite{el2010information} is more general in another direction: it applies %
	to $X_i = \mu_i + \sqrt\Sigma W_i$ where $W_i$ have independent coordinates with bounded fourth moments. 
	(Note that $\Sigma$ is the same for all data points in~\cite{el2010information}.) 
	Since we seek exponential concentration, we need stronger control of the tail probabilities. \qed
\end{rem}
\begin{exa}[Gaussian kernel and isotropic noise]\label{exa:Gauss:Gauss}
	Let us consider the implications of Theorem~\ref{thm:Lip:ker:gen} for the Gaussian kernel, assuming that the underlying random vectors follow:
	\begin{align}\label{eq:simp:Gauss:model}
	X_i = \mu_i + \frac{\sigma_i}{\sqrt d} \,w_i, \quad w_i \iid N(0,I_d).
	\end{align}
	As will be discussed in Section~\ref{sec:cluster}, by allowing $\mu_i$ to vary over some latent clusters in the data, \eqref{eq:simp:Gauss:model} provides a simple model for studying clustering problems. The scaling of the noise variances by $\sqrt{d}$ is so that the two terms $\mu_i$ and $(\sigma_i/\sqrt d) w_i$ are balanced in size as $d \to \infty$. Without the scaling, since $\norm{w_i}$ concentrates around $\sqrt{d}$, the noise $\sigma_i w_i$ will wash out the information in the signal $\mu_i$ (assuming $\norm{\mu_i} = O(1)$ as $d \to \infty$).
	
	Consider the Gaussian kernel function on $(\reals^d)^2$ with bandwidth parameter $\tau$:
	\begin{align}\label{eq:Gauss:ker}
	K(x,y) = \exp\Big({ -\frac1{2\tau^2}} \norm{x-y}^2\Big) = f_\tau(\norm{x-y}), \quad f_\tau(t) := e^{-t^2/2\tau^2}.
	\end{align}
	This is  a Lipschitz kernel with $L = \infnorm{f_\tau'} = \sqrt2 / (e\tau)$. The expected kernel matrix $\ex K$ has the following entries (see Appendix~\ref{app:Gauss:ker:comp}):
	\begin{align*}
	[\ex K]_{ij} = \frac1{s_{ij}^d}\exp\Big(  {-\frac{\norm{\mu_i - \mu_j}^2}{2s_{ij}^2 \tau^2}} \Big), \quad s_{ij}^2 = 1+ \frac{\sigma_i^2 + \sigma_j^2}{d\tau^2}, \quad i \neq j.
	\end{align*}

	Consider the special case where $\sigma_i = \sigma$ for all $i$, and let $s^2 = 1 + 2\sigma^2/(d\tau^2)$. %
	Then,
	the mean kernel matrix $\ex K$ is itself a kernel matrix,  based on a Gaussian kernel with updated bandwidth parameter $\tau s$, applied to mean vectors $\{\mu_i\}$, that is, 
	\begin{align*}
		\Kt_\sigma(\mu_i,\mu_j) := [\ex K]_{ij} = s^{-d}f_{\tau s}(\norm{\mu_i - \mu_j}).
	\end{align*}
	Note that the mean kernel matrix depends on the noise variance $\sigma$. Also, because of the scaling of the variance in~\eqref{eq:simp:Gauss:model}, the prefactor $s^{-d}$ stabilizes as $d \to \infty$, that is, $s^{-d} = (1+2\sigma^2/(d\tau^2))^{-d/2} \to e^{-\sigma^2/\tau^2}$ and the kernel function approaches the standard Gaussian kernel $f_{s\tau} \to f_1$. (Without the variance scaling, the prefactor would go to zero.)  
	
	\smallskip
	Applying~\eqref{eq:dist:ker:Gauss} with $\siginf = \sigma$, $\omega=1$, $c=1/2$, $L=\sqrt2 / (e\tau)$ and replacing $t$ with $\sqrt{2} t$, %
	\begin{align}\label{eq:Gauss:Gauss:bound}
		\frac1n \opnorm{K - \ex K} \le \frac{4}{e} \frac{\sigma}{\tau} \frac1{\sqrt d} \Big( 1 + \frac{t}{\sqrt{n}}\Big), \quad \text{w.p.} \; \ge 1- e^{-t^2}.
	\end{align}
	It is interesting to note that the deviation is controlled by the ratio $\sigma/\tau$.
	For example, we could have started with the alternative model without the scaling of the standard deviation by $\sqrt{d}$, that is, model~\eqref{eq:simp:Gauss:model} with $\sigma_i/\sqrt{d}$ replaced with $\sigma$, but instead rescaled the bandwidth by changing $\tau$ to $\tau \sqrt{d}$. Then, we would have the same exact concentration bound as in~\eqref{eq:Gauss:Gauss:bound}. This observation somewhat justifies the rule of thumb used in practice where one sets the bandwidth $\propto \sqrt{d}$ in the absence of additional information. According to the above discussion, this choice roughly corresponds to the belief that the per-coordinate standard deviation is $O(1)$ as $d \to \infty$.
	\qed
\end{exa}

Example~\ref{exa:Gauss:Gauss} can be easily extended to the case of anisotropic noise, using the invariance of  both the Gaussian kernel and the Gaussian distribution to unitary transformations. More generally, consider an extension of model~\eqref{eq:simp:Gauss:model} as follows
	\begin{align}\label{eq:gauss:aniso:model}
		X_i = \mu_i + \frac{1}{\sqrt d} \,w_i, \quad w_i \iid N(0,\Sigma).
	\end{align}
	This is similar to the model in~\cite{el2010information}, assuming in addition the Gaussianity of the noise.  Applying~\eqref{eq:dist:ker:Gauss}, %
	replacing $\Sigma$ with $\Sigma/\sqrt{d}$,  we have for model~\eqref{eq:gauss:aniso:model},
	\begin{align}
		\label{eq:lip:ker:concent:anisotropic}
		\frac1n 	\opnorm{K - \ex K} \le \; 2 \sqrt{2} L \sqrt{\frac{\opnorm{\Sigma}}d}
		 \Big( 1 + \frac{t}{\sqrt{n}}\Big), \quad \text{w.p.} \; \ge 1- e^{-t^2}.
	\end{align}
	In practice, it is often reasonable to assume $\opnorm{\Sigma} = O(1)$.
	Then, $	\frac1n\opnorm{K - \ex K}  = O_p(d^{-1/2})$ as $d \to \infty$, that is, we get consistency in estimating $\ex (K/n)$ by $K/n$, as dimension $d$ grows.

\subsection{Euclidean kernel }\label{sec:Euclid:Ker}
We now consider the kernel function $K(x_1,x_2) = \ip{x_1,x_2}$ which we refer to as the \emph{Euclidean} or \emph{inner product} kernel. The kernel matrix in this case is the Gram matrix of $\{X_i\}$:
\begin{align}\label{eq:ip:ker}
	K(X) = (\ip{X_i,X_j}) = X^T X.
\end{align}
Our main result for the Euclidean kernel is the following:

\begin{thm}\label{thm:ip:ker}
	Let $X_i = \mu_i + \sqrt{\Sigma_i} W_i$, where $\{W_i, i=1,\dots,n\} \subset \reals^d$ is a collection of independent centered random vectors, each with independent sub-Gaussian coordinates. Here, $\mu_i = \ex[X_i] \in \reals^d$ and each $\Sigma_i$ is a $d \times d$ positive semidefinite matrix, with generalized square root $\sqrt{\Sigma_i}$. Let
	\begin{align*}
	M &= (\mu_1 \mid \cdots \mid \mu_n) \in \reals^{d \times n}, \quad 
	\kappa = \max_{i,j} \sgnorm{W_{ij}},  %
	  \\
	\siginf^2 &:= \max_i \opnorm{\Sigma_i}, \qquad  \eta = d + \Big( \frac{\opnorm{M}}{\kappa \,\siginf}\Big)^2.
	\end{align*}
	For $K = K(X)$ as in~\eqref{eq:ip:ker} and for any $u \ge 0$, with probability at least $1 - 4 n^{-c_1}\exp(-c_2 u^2)$, 
	\begin{align*}
	\opnorm{K - \ex K}\le 2 \kappa^2 \siginf^2 \eta \, \max(\delta^2,\delta), \quad \text{where} \; 
	\delta = \sqrt{\frac{ n}{\eta}} + \frac{u}{\sqrt{\eta}}.
	\end{align*}
	In particular, with probability at least $1-4 n^{-c_1}$,
	\begin{align}\label{eq:ip:ker:concent:gen}
	\begin{split}
	\opnorm{K - \ex K} &\;= \; O\Bigl(\kappa^2 \siginf^2 (n + \sqrt{n\eta}) \Bigr) \\
	&\;= \;  O\Bigl(\kappa^2 \siginf^2(n+ \sqrt{nd}) + \kappa\siginf\sqrt{n}   \opnorm{M}\Bigr).
	\end{split}
	\end{align}
\end{thm}
A special case of this result, when $X_i$s are centered and isotropic ($\mu_i = 0$, $\Sigma_i = I_d$ and $\ex W_{ij}^2 = 1$ for all $i$ and $j$), appears in~\cite[Section~5.5]{vershynin12}.
The normalized $n \times n$  kernel matrix $\frac1n X^T X$ is dual to the $d \times d$ matrix $\frac1n X X^T = \frac1n \sum_{i=1}^n X_iX_i^T$ which is the main component of the sample covariance matrix of $\{X_i\}$. Thus, Theorem~\ref{thm:ip:ker} is  dual to the well-known concentration results for covariance matrices.
However, a major difference with covariance estimation is that with Gram matrices, the data points need not have identical distributions.

An interesting feature of  bound~\eqref{eq:ip:ker:concent:gen} is its dependence on the mean of the underlying vectors through $\opnorm{M}$. Contrast this with the result of Theorem~\ref{thm:Lip:ker:gen} %
where the  bound is not affected by the mean of the random vectors $X_i$. 
Under the assumptions of Theorem~\ref{thm:ip:ker},
the mean kernel matrix is $\ex K = \diag(\ex \norm{\Xt_i}^2,i \in [n]) + M^T M$, where $\Xt_i = X_i - \mu_i$ is the centered version of $X_i$. The second term has operator norm $\opnorm{M^T M} = \opnorm{M}^2$, whereas the relevant term in~\eqref{eq:ip:ker:concent:gen} is of lower order in $\opnorm{M}$. More precisely, $\frac{\opnorm{K - \ex K}}{\opnorm{\ex K}}  \lesssim \frac{1}{\opnorm{M}}$ as $\opnorm{M} \to \infty$, confirming that~\eqref{eq:ip:ker:concent:gen} is indeed a concentration result.

\begin{exa}
	Let us continue with model~\eqref{eq:simp:Gauss:model} of Example~\ref{exa:Gauss:Gauss}. The model corresponds to $\Sigma_i = \sigma_i^2 I_d / d$ and $\kappa \lesssim 1$ in Theorem~\ref{thm:ip:ker}.  Assume that $\sigma_i \le \sigma$ for all $i$. It follows that $\sigma_\infty \le \sigma / \sqrt{d}$ and Theorem~\ref{thm:ip:ker} gives
	\begin{align*}
		\frac1n \opnorm{K - \ex K} \;\lesssim\; \sigma^2 \Big( \frac1{d} + \frac1{\sqrt{n d}}\Big)
		+\frac\sigma{\sqrt{nd}} \opnorm{M}, \quad \text{w.p.} \; \ge 1-4 n^{-c_1}.
	\end{align*}
	Compared with~\eqref{eq:Gauss:Gauss:bound}, the deviation bound improves as $d$ is increased. On the other hand, the bound directly depends on the mean matrix $M = \ex X$, as opposed to~\eqref{eq:Gauss:Gauss:bound}. \qed
\end{exa}
The bound in~\eqref{eq:ip:ker:concent:gen} is sharp in general. To see this, first consider the term $\kappa^2 \siginf^2(n+ \sqrt{nd})$. Without loss of generality, assume $\siginf^2 = 1$. Consider the case $X_i \sim N(0,I_d)$, drawn i.i.d., and let $y_k$ be the $k$-th row of $(X_1 \mid \cdots \mid X_n)$. Then, $y_k, k=1,\dots,d$ are i.i.d. draws from $N(0,I_n)$. Hence, $\frac1d \norm{K - \ex K} = \norm{\frac1d \sum_k y_i y_i^T - I_n}$ is the deviation of a sample covariance matrix from its expectation which is known to scale as $\sqrt{\frac{n}{d}} + \frac{n}{d}$. See for example~\cite[Theorem~4.7.1]{HDP}.

The last term in~\eqref{eq:ip:ker:concent:gen} is also unavoidable when $n \ge C d$ for a sufficiently large constant $C$. To see this, let $X = \siginf^{-1} M + \siginf W \in \reals^{d\times n}$ where $X_i$, $M_i$ and $W_i$ are the $i$th columns of $X$, $M$ and $W$, respectively, and $W_i \sim N(0,I_d)$ drawn i.i.d. Letting $\siginf \to 0$, we have $\norm{K - \ex K} \to 2 \opnorm{M^T W}$. Note that $\frac1n W W^T$ is a sample covariance matrix, concentrated around $I_d$. By taking $n \ge C d$ for a large constant $C$, we have $\frac1n W W^T \succeq \frac12 I_d$, with high probability. It follows that $2\opnorm{M^T W} = 2\sqrt{\opnorm{M^T W W^T M}} \ge 2 (\frac{n}2 \opnorm{M^T M})^{1/2} \ge \sqrt{2n} \opnorm{M}$, which is proportional to bound~\eqref{eq:ip:ker:concent:gen} after replacing $M$ with $\siginf^{-1} M$ and letting $\siginf \to 0$.

\section{Kernel spectral clustering}\label{sec:cluster}
We now consider how the concentration bounds of Section~\ref{sec:concent:ineq} can be used to derive performance bounds for the kernel spectral clustering. 

\subsection{A kernel clustering algorithm}
Let $\mu \mapsto \Sigma(\mu)$ be a map from $\reals^d$ to positive semidefinite matrices, and let $ \sqrt{\Sigma(\mu)}$ denote its matrix square-root. We consider a nonparametric mixture model perturbed by noise, as follows: %
\begin{align}\label{eq:nonparam:mix:noise:2}
X_i = \mu_i + \frac{\sigma}{\sqrt d} \sqrt{\Sigma(\mu_i)} w_i, \quad \mu_i \iid \sum_{k=1}^\numc \pib_k P_k,  \quad w_i \iid N\big(0,I_d\big),
\end{align}
for $i=1,\dots,n$, where $\mu_i$ is the signal, $w_i$ is the noise, and the two pieces are independent. Note that the distribution of $X_i$ goes beyond a nonparametric mixture model unless $\mu \mapsto \Sigma(\mu)$ is constant.
 The Gaussian assumption for $w_i$ is for simplicity; the result holds for all the cases in~Theorem~\ref{thm:Lip:ker:gen}. %
 Here, $\{P_k\}$ are the distributions constituting the mixture components, and $\pib_k \in [0,1]$ are the class priors.
In a typical case, components $\{P_k\}$ are supported on lower-dimensional sub-manifolds of $\reals^d$, singular w.r.t. the Lebesgue measure and singular w.r.t. to each other; see for example Figure~\ref{fig:scatters}. Although, none of these assumptions are required for the result we present. Intuitively, the kernel clustering should perform well if we only observe $\{\mu_i\}$ and we would like to study the effect of adding noise to such ideal clustered data.

Model~\eqref{eq:nonparam:mix:noise:2} is sufficiently general to allow the noise structure to vary based on the signal. A special case is when $\Sigma(\mu) = \Sigma_0$ is constant, in which case the model is equivalent to
\begin{align}\label{eq:nonparam:mix:noise}
X_i = \mu_i + \frac{\sigma}{\sqrt d} w'_i, \quad \mu_i \iid \sum_{k=1}^\numc \pib_k P_k,  \quad w'_i \iid N(0,\Sigma_0).
\end{align}
This special case is often encountered in the literature.

\begin{figure}[t]
	\centering
	\includegraphics[width=.9\textwidth]{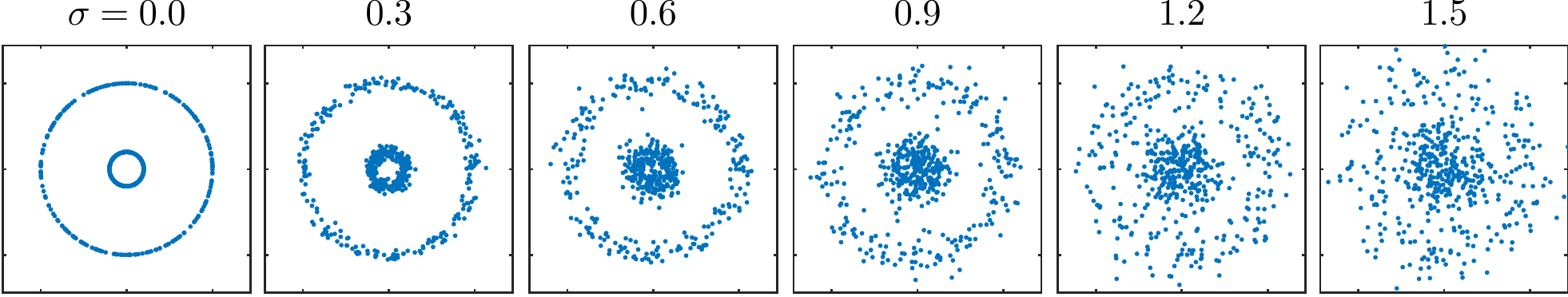}
	\caption{Example of the signal-plus-noise clustering model~\eqref{eq:nonparam:mix:noise} with two signal component $P_1$ and $P_2$, each a uniform distribution on a circle in $d=2$ dimensions, and $\Sigma_0 = I_2$. The plots correspond to different noise levels $\sigma$.}
	\label{fig:scatters}
\end{figure}

Given a kernel function, we can form the kernel matrix $K=K(X)$ as  in~\eqref{eq:K(X)}.  Throughout this section, unless otherwise stated, we condition on $\mu = (\mu_i)$, hence the expectations and probability statements are w.r.t. the randomness in $w = (w_i)$.  Let $\Kt_\sigma(\mu) := \ex [K(x_i,x_j)]$, which should be interpreted as $\Kt_\sigma(\mu) = \ex [K(x_i,x_j) \mid \mu]$, by the convention just discussed. The mean kernel matrix $\Kt(\mu)$ has the following off-diagonal entries under model~\eqref{eq:nonparam:mix:noise:2}:
\begin{align}\label{eq:mean:ker:mat}
[\Kt_\sigma(\mu)]_{ij} &= [\ex K]_{ij} = \Kt_\sigma(\mu_i,\mu_j), \quad i \neq j,
\end{align}
where, with some abuse of the notation regarding $\Kt_\sigma$, we have defined:
\begin{align}\label{eq:mean:kern:func:2}
\Kt_\sigma (u,v) := \ex\Big[K\Big(u+\frac{\sigma}{\sqrt{d}}\sqrt{\Sigma(u)} \,w_1,\,v+ \frac{\sigma}{\sqrt{d}} \sqrt{\Sigma(v)} \,w_2\Big)\Big], \quad u \neq v.
\end{align}
Here, the expectation is w.r.t. the randomness in $w_1$ and $w_2$. Note that we are using $\Kt_\sigma$ to refer to both the mean kernel matrix and the corresponding kernel function. In the special case of constant noise covariance, $\Sigma(\mu) = \Sigma_0$, we simply have
\begin{align}\label{eq:mean:kern:func}	
\Kt_\sigma (u,v) &:= \ex\Big[K\Big(u+\frac{\sigma}{\sqrt{d}}w'_1,\,v+ \frac{\sigma}{\sqrt{d}} w'_2\Big)\Big], \quad u \neq v,
\end{align}
where $w'_1$ and $w'_2$ are independent $N(0,\Sigma_0)$ variates. The properties of the new kernel matrix $\Kt_\sigma(\mu)$ plays a key role in our analysis.

\begin{algorithm}[t]
	\caption{A kernel spectral clustering (KSC) algorithm}
	\label{alg:kern:clust}
	\begin{algorithmic}[1]
		\linespread{1.2}\selectfont
		\Require	
		(a) Data points $x_1,\dots,x_n \in \reals$, 
		(b) the number of clusters $\numc$ and 
		(c) the kernel function $(x,y) \mapsto K(x,y)$, not necessarily positive semidefinite.
		
		\Ensure  Cluster labels.
		\State Form the normalized kernel matrix $A := ( K(x_i,x_j) / n) \in \reals^{n \times n}$.
		\State Obtain $A^{(\numc)} = \Uh_1  \Lamh_1 \Uh_1^T$, the $\numc$-truncated eigenvalue decomposition (EVD) of $A$. That is, if $A = \Uh \Lamh \Uh^T$ is the full EVD of $A$, where $\Lamh = \diag(\lamh_1,\dots,\lamh_n)$ with $|\lamh_1| \ge \dots \ge |\lamh_n|$, then $\Lamh_1 = \diag(\lamh_1,\dots,\lamh_\numc)$, and $\Uh_1 \in \reals^{n \times \numc}$ collects the first $\numc$ columns of $\Uh$.
		\State Apply an isometry-invariant, constant-factor, \kmeans algorithm (with $\numc$ clusters) on $\Uh_1 \Lamh_1$ to recover the cluster labels. 
	\end{algorithmic}
	
\end{algorithm}
We analyze the kernel-based spectral clustering (KSC) approach summarized in Algorithm~\ref{alg:kern:clust} which is based on the recent SC-RRE algorithm of~\cite{zhou2019analysis} for network clustering. An advantage of this spectral algorithm is that we can provide theoretical guarantees that are explicitly expressed in terms of the original parameters of the model, avoiding eigenvalues in the statement of the bounds. The connection with network clustering is as follow: We can treat $K/n \in \reals^{n \times n}$ as a similarity matrix, effectively defining a weighted network among $n$ entities, and then use the adjacency-based spectral clustering described in~\cite{zhou2019analysis}.

Algorithm~\ref{alg:kern:clust} proceeds by forming the $R$-truncated eigenvalue decomposition of the similarity matrix $A= K/n$, denoted as $A^{(\numc)} = \Uh_1  \Lamh_1 \Uh_1^T$. One then performs a constant-factor approximate $k$-means algorithm on the rows of $\Uh_1 \Lamh_1$ to obtain the estimated cluster labels. The details of this step are as follows: For a set $\Yc = \{y_1,\dots,y_\numc\} \subset \reals^{D}$ and any point $x \in \reals^D$, let $d(x,\Yc) = \min_{y \in \Yc}\norm{x-y}$. The $k$-means problem, with $R$ clusters, seeks to minimize $\sum_{i=1}^n d(X_i,\Yc)^2$ over $\numc$-element subsets $\Yc$ of $\reals^D$. This problem is in general NP-hard. However, it is possible to find $\kappa$-approximate solutions in polynomial-time, i.e., $\widehat\Yc$ such that $\sum_i d(X_i,\widehat \Yc)^2 \le \kappa \cdot \min_{\Yc} \sum_{i} d(X_i,\Yc)^2$. Given, $\widehat \Yc$, every point $X_i$ is mapped to the closest element of $\widehat \Yc$, producing cluster labels. We further assume that the algorithm for deriving the $\kappa$-approximate solution is isometry-invariant, that is, it only depends on the pairwise distances among $\{X_i\}$. Examples of such algoirthms for deriving a $\kappa = 1+\eps$ approximation are the approach of~\cite{kumar2004simple} with time complexity $O(2^{\text{poly}(R/\eps)} n D)$~\cite{ackermann2010clustering} and that of~\cite{chen2009coresets} with complexity $O(n D R + 2^{\text{poly}(R/\eps)} D^2 \log^{D+2} n)$. Since we apply these algorithms with $\eps = O(1)$ and $D = R$, assuming $R = O(1)$, both algorithms run in $O(n)$ time.

\subsection{Finite-sample bounds on misclassification error}
\label{sec:finite-sample:mis}

 Let $z_i \in \{0,1\}^\numc$ be the label of data point $i$, determining the component of the mixture %
 to which $\mu_i$ belongs. We use one-hot encoding for $z_i$, so that $z_{ik} = 1$ if and only if data point $i$ belongs to cluster $k$, that is, %
  $\mu_i \sim P_k$. %
  Let $\Cc_k := \{i : z_{ik} = 1\}$ denote the indices of data points in the $k$th cluster, $n_k := |\Cc_k|$ and $\pi_k := n_k / n,$
  the size and the (empirical) proportion of the $k$th cluster, respectively.

For $k,\ell \in [\numc]$, let $\Ph_{k,\ell}$ be the empirical measure on $\reals^d \times \reals^d$ given by
\begin{align*}
	\Ph_{k\ell} := \Ph_{k\ell}(\mu) = \frac1{n_k n_\ell} \sum_{(i,j) \,\in\, [n]^2}  z_{ik} z_{j\ell} \, \delta_{(\mu_i,\mu_j)} = 
	\frac1{n_k n_\ell} \sum_{i \,\in\, \Cc_k, \,j \,\in \,\Cc_\ell}\delta_{(\mu_i,\mu_j)}
\end{align*}
where $\delta_{(\mu_i,\mu_j)}$ is a point-mass measure at $(\mu_i, \mu_j)$. In words, $\Ph_{k\ell}$ is the empirical measure %
when the data consists of pairs $(\mu_i,\mu_j)$, as $i$ and $j$ range over the $k$th and $\ell$th clusters, respectively.
Consider the mean and variances of these empirical measures:
\begin{align}\label{eq:Psi:def}
	\Psi_{k\ell} :=  \ex \big[\Kt_\sigma(X,Y)\big], \quad 	v_{k\ell}^2 := %
	\var\big( \Kt_\sigma(X,Y) \big) \quad \text{where} \quad (X,Y) \sim \Ph_{k\ell}.
\end{align}
Let $\vb^2$ be the average variance
\begin{align}\label{eq:vb:def}
	\vb^2 := %
	\sum_{k,\,\ell \,\in\, [\numc]} \pi_k \pi_\ell\, v_{k\ell}^2,
\end{align}
and define the following minimum separations: %
\begin{align}\label{eq:gam}
	\gamma^2 := \min_{k \neq \ell} D_{k\ell}, \quad \gamt^2 := \min_{k \neq \ell} \pi_\ell  D_{k\ell}, \quad \text{where} \; D_{k\ell} := \sum_{r=1}^\numc \pi_{r}(\Psi_{kr} - \Psi_{\ell r})^2.
\end{align}
When the clusters are roughly balanced, %
we have $\pi_k \asymp 1/\numc$ for all $k \in [\numc]$, hence $\gamt^2 \asymp  \gamma^2 / \numc$. If the number of clusters does not grow with $n$, %
then $\gamt^2\asymp \gamma^2$. %

Let $\{\zh_i\}$ be the labels outputted by Algorithm~\ref{alg:kern:clust} and let $\Misb$ be the corresponding average misclassification rate relative to the true labels. That is, $\Misb = \min_{\sigma} \frac1n 1\{\sigma(\zh_i) \neq z_i \}$ where the minimum is take over all permutations $\sigma: [\numc] \to [\numc]$. (Here, we treat both $\zh_i$ and $z_i$ as elements of $[\numc]$.)
We are now ready to state our result on the performance of kernel spectral clustering:
\begin{thm}\label{thm:mis}
	Assume that the data points $\{X_i, i=1,\dots,d\} \subset \reals^d$ follow the nonparametric noisy mixture model~\eqref{eq:nonparam:mix:noise:2}. 
	Consider the kernel spectral clustering Algorithm~\ref{alg:kern:clust} with an $L$-Lipschitz kernel function as in~\eqref{eq:Lip:kern:ineq}. Let $\vb^2$ and $\gamma^2$ be defined, based on $\Kt_\sigma$, as given in~\eqref{eq:mean:kern:func:2}. Fix $t \ge 0$, and let
	\begin{align}\label{eq:F:def}
	    F(\gamma^2, \vb^2) := \frac{16\numc}{\gamma^2}
	     \Big[ \frac{4 L^2 \sigma^2}{d}\Big( 1 + \frac{t}{\sqrt n}\Big)^2 \max_i \opnorm{\Sigma(\mu_i)}
	     + \vb^2 \Big] 
	\end{align}
	and $C_1 := 4(1+\kappa)^2$ where $\kappa$ is the approximation factor of the $k$-means algorithm. Assume that  $F(\gamt^2,\vb^2) \le C_1^{-1}$.
	Then, with probability at least $1-\exp(-t^2)$, the average misclassification rate of Algorithm~\ref{alg:kern:clust}
	satisfies
	\begin{align}\label{eq:mis:bound}
	\Misb \le C_1 F(\gamma^2, \vb^2).
	\end{align}

\end{thm}
A similar result can be stated for the Euclidean kernel of Section~\ref{sec:Euclid:Ker}. 
Consider the special case where $\Sigma(\mu) = \Sigma_0$ for all $\mu$.
 The quantity $\vb^2 / \gamma^2$ in~\eqref{eq:F:def} is a measure of the hardness of the noiseless clustering problem, which we  refer to as the approximation error. The first term in the bound~\eqref{eq:F:def} is the contribution due to noise, %
 the so-called estimation error. Both  quantities depend on the noise level $\sigma$ as well as the noise structure $\Sigma_0$, through $\Kt_\sigma$ in~\eqref{eq:mean:kern:func}.
 Thus, a more precise statement is that $\vb^2 / \gamma^2$ measures the hardness of the noiseless problem \emph{at the appropriate level determined by the noise level $\sigma$ (and noise structure $\Sigma_0$)}. This dependence on noise can become negligible in the high-dimensional setting where $d \to \infty$; see Section~\ref{sec:pop:params}.

In addition, both of the terms depend on the choice of the kernel function $K(\cdot,\cdot)$: the estimation error through the Lipschitz constant $L$ and approximation error clearly as the definitions of $\vb^2$ and $\gamma^2$ show. When the kernel class has a tuning parameter, one might be able to trade-off the contributions of these terms as the following example shows.

\begin{exa}[Spectral clustering with Gaussian kernel]\label{exa:sc:isotropic}
	Consider the case of constant isotropic noise $\Sigma_0 = I_d$ and the  Gaussian kernel~\eqref{eq:Gauss:ker} with bandwidth~$\tau$. As discussed in Example~\ref{exa:Gauss:Gauss}, the Lipschitz constant is $L \lesssim 1/ \tau$. Thus the misclassification bound~\eqref{eq:F:def} in this case reduces to
	\begin{align}\label{eq:miss:isotropic:exa}
		\Misb \; \lesssim \; \frac{\numc}{\gamma^2} \Big[\frac{\sigma^2}{\tau^2} \frac1d\Big( 1 + \frac{t}{\sqrt{n}} \Big)^2 + \vb^2 \Big]
	\end{align}
	which holds with probability $\ge 1-e^{-t^2}$. 
	Roughly speaking, assuming $R=O(1)$, the estimation error is $\lesssim \sigma^2 / (\gamma^2 \tau^2 d)$  and the approximation error $\lesssim \vb^2/\gamma^2$. The estimation error is $O(d^{-1})$, as $d \to \infty$, assuming that $\gamma^2$ stays away from 0, which is the case as discussed in Section~\ref{sec:pop:params}.
	
	As argued in Example~\ref{exa:Gauss:Gauss}, $\Kt_\sigma$ is again a Gaussian kernel, with modified bandwidth:
	\begin{align}\label{eq:Kt:s:def:isotropic}
		\Kt_\sigma(\mu_i,\mu_j) = s^{-d} f_{\tau s}( \norm{\mu_i - \mu_j}), \quad s^2 = 1 + \frac{2 \sigma^2}{d \tau^2}.
	\end{align}
	Since $\vb^2$ and $\gamma^2$ are defined based on $\Kt_\sigma$, both the approximation and estimation errors depend on the normalized bandwidth $\tau /\sigma$. In addition, the approximation error also depends on the bandwidth-normalized pairwise distances of the signal component, i.e., $ \norm{\mu_i - \mu_j}/{\tau}$, for $i,j \in [n]$. It is interesting to note that the dependence of the approximation error on the noise level $\sigma$ vanishes as $d \to \infty$.
	In Example~\ref{exa:concent:iso:approx} below, we provide explicit limit expressions for $\vb^2$ and $\gamma^2$.
	\qed
\end{exa}
\begin{figure}
	\centering
	\includegraphics[width=\textwidth]{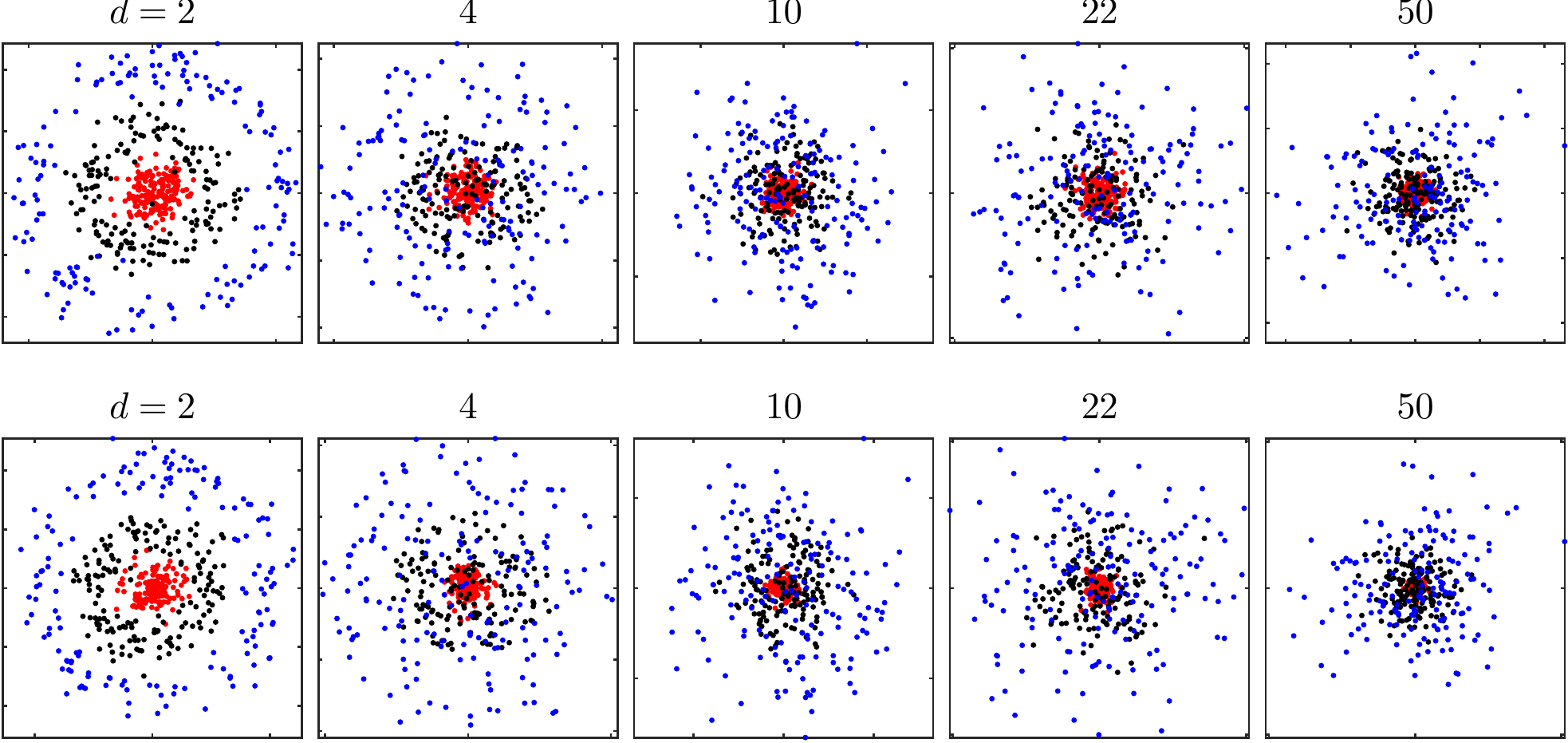}
	\caption{Plots of the first two coordinates of $X_i$ for the ``nested spheres'' example, with radii $r_i = 1,5,10$, noise level $\sigma = 1.5$ and variable $d$. The top and bottom rows corresponds to the isotropic versus radial noise models, respectively. The plots look qualitatively the same in both cases. As can be seen, for large $d$, it is very hard to distinguish the clusters from a low-dimensional projection. (The scale of the plots varies with $d$.)}
	\label{fig:aniso:model}
\end{figure}
\begin{exa}[Nested spheres with radial noise]\label{exa:spheres:anisotrop}
	 Assume that the signal mixture components $\{P_k\}$ are uniform distributions on nested spheres in $\reals^d$ of various radii: $r_1,\dots,r_\numc$. Assume that to each $\mu_i$, drawn from the mixture, we add a Gaussian noise in the direction perpendicular to the sphere, i.e.,
\begin{align}
	X_i = \mu_i + \frac{\sigma}{\sqrt{d}} \frac{\mu_i}{\norm{\mu_i}} \xi_i, \quad \xi_i \iid N(0,1).
\end{align}
This noise structure falls under model~\eqref{eq:nonparam:mix:noise:2} with $\Sigma(\mu) := \mu \mu^T/ \norm{\mu}^2$, i.e., the rank-one projection onto the span of $\mu$. Since $\max_i \opnorm{\Sigma(\mu_i)} = 1$, the missclassification bound obtained from~\eqref{eq:F:def} is similar to~\eqref{eq:miss:isotropic:exa} in the isotropic case, with $1/\tau^2$ replaced with $L^2$.
Thus, 
the dominant term in the estimation error is $\lesssim (R L^2 \sigma^2)/ (\gamma^2 d)$ which is $O(1/d)$ as $d \to \infty$, assuming that $\gamma^2$ stays bounded away from $0$. (This is the case as discussed in Section~\ref{sec:pop:params}.) Note that the behavior of the estimation error is  the same as that of the isotropic case. 
Let us also compute the mean kernel function, assuming as the base, the usual Gaussian kernel~\eqref{eq:Gauss:ker}. We have
\begin{align*}
	\Kt_\sigma (u,v) &:= \ex\Big[K\Big(u+\frac{\sigma}{\sqrt{d}} \ut \,\xi_1,\,v+ \frac{\sigma}{\sqrt{d}} \vt \,\xi_2\Big)\Big] \\ &= \ex \exp\Big( {-}\frac{\norm{u - v + (\sigma/\sqrt{d}) w}^2}{2\tau^2}\Big)
\end{align*}
where $\ut = u/\norm{u}$, $\vt = v/\norm{v}$, and $w = \ut \xi_1 - \vt \xi_2 \sim N(0, \ut\ut^T + \vt \vt^T)$. One can show that
\begin{align}\label{eq:Kt:def:anisotrop}
\begin{split}
	\Kt_\sigma (u,v) &= \frac{1}{s_1 s_2}\exp\Big\{ 
		{-}\frac1{2\tau^2} \Big[
		\frac{\lambda_1}{2s_1^2} \big( \norm{u} - \sign(\alpha) \norm{v}\big)^2 \\ 
		&\qquad \qquad \qquad\qquad\qquad + \frac{\lambda_2}{2s_2^2} \big( \norm{u} + \sign(\alpha) \norm{v}\big)^2 \Big]  \Big\}, \\
	s_i^2 &= 1+ \frac{\sigma^2 \lambda_i}{\tau^2 d},\; i=1,2, \\
	\lambda_1 &= 1+|\alpha|, \quad \lambda_2 = 1-|\alpha|, \quad \alpha = \frac{\ip{u,v}}{\norm{u} \norm{v}}
\end{split}
\end{align}
assuming that $\alpha \neq 0$, and $u \neq v$. See Appendix~\ref{sec:details:Ker:aniso} for details. It is interesting to note that this mean kernel mostly depends on the norms of $u$ and $v$.  The dependence on $\alpha$, the angle between $u$ and $v$, is quite weak (through $s_i^2$ and $\sign(\alpha)$) and mostly goes away as $d \to \infty$. In the next section, we argue that the approximation error $\vb^2/\gamma^2$ based on this kernel also goes to zero as $n, d \to \infty$. \qed
\end{exa}

\subsection{Population-level parameters}\label{sec:pop:params}
The quantities $v_{k\ell}^2$ and $\Psi_{k\ell}^2$ that underlie $\vb^2$ and $\gamma^2$, and control the approximation error in Theorem~\ref{thm:mis}, are defined based on the empirical measures $\Ph_{k\ell}$. But it is also possible to state them directly in terms of the underlying population-level components $\{P_k\}$ and the related integrals. The main idea is that $\Ph_{k\ell}$, in general, has a well-defined limit:
\begin{align}\label{eq:conv:Pkl}
	\Ph_{k\ell} \to P_k \otimes P_\ell, \quad \text{as} \quad n \to \infty, \quad \text{w.h.p.}
\end{align}
where the convergence can be interpreted in various senses (e.g. weak convergence of probability measures, or convergence in $L^p$ Wasserstein distances). The notation $ P_k \otimes P_\ell$ represents a product measure, i.e., if $(X,Y) \sim P_k \otimes P_\ell$, then $X$ and $Y$ are independent variables with marginal distributions $P_k$ and $P_\ell$.  The convergence in~\eqref{eq:conv:Pkl} holds even when $k=\ell$ (cf. Proposition~\ref{prop:concent:Psi:v} below).
Let 
\begin{align*}
	 \Psi_{k\ell}^* :=  \int \Kt_\sigma(\mu,\mu') \,dP_k(\mu) \,dP_\ell(\mu'), \quad (v_{k\ell}^*)^2 := \var\big( \Kt_\sigma(X,Y) \big)
\end{align*}
 where $(X,Y) \sim P_k \otimes P_\ell$.
 Similarly, let $D_{k\ell}^*$, $\gams^2$ and $\vbs^2$ be the population-level versions of $D_{k\ell}$, $\gamma^2$ and $\vb^2$ obtained by replacing $\Psi_{k\ell}$ and $v_{k \ell}^2$ with their starred versions in the corresponding definitions.
  The above discussion suggests that for large $n$, $\Psi_{k\ell} \approx \Psi_{k\ell}^*$ and $v_{k\ell}^2 \approx (v_{k\ell}^*)^2$ and similarly for the other related quantities. The following result formalizes these ideas: 
 \begin{prop}\label{prop:concent:Psi:v}
 	Assume that $\Kt_\sigma$ has constant diagonal and is uniformly bounded on the union of the supports of $P_k, k \in [R]$, so that $|\Kt_\sigma(\mu_i,\mu_j)| \le b$ a.s. for all $i,j \in [n]$ and some $b > 0$.    Then, with probability at least $1- 4 R^2 \exp(- t^2)$, for all $k, \ell \in [R]$,
 	\begin{align*}
 		|\Psi_{k\ell} - \Psi^*_{k\ell}| \le   \frac{3  b t}{\sqrt{n_k \wedge n_\ell}} =: \delta_{k\ell}, \quad |v_{k\ell}^2 - (v_{k\ell}^*)^2| \le \frac{ 9   b^2 t}{\sqrt{n_k \wedge n_\ell}}. %
 	\end{align*}
 	Letting $\pimin = \min_{k} \pi_k$,  on the same event, we have
 	\begin{align}\label{eq:gamma:vb:approx}
 		\gamma^2 \ge \gams^2 - \frac{24 b^2 t}{ \sqrt{\pimin}} \frac1{\sqrt n}, \quad 
 		\vb^2 \le \vbs^2 + \frac{ 9   b^2 t}{\sqrt{\pimin}} \frac1{\sqrt n}.
 	\end{align}
 \end{prop}
 Note that the bounds in~\eqref{eq:gamma:vb:approx} are dimension-free: Assume that $\pimin$ is bounded below. Then, as long as $n$ is sufficiently large, both $\gams^2$ and $\vbs^2$ are good approximations for their empirical versions, irrespective of how large $d$ is.  When $\gams^2$ is bounded below, we can replace $\gamma^2$ and $\vb^2$ in the misclassification bound in Theorem~\ref{thm:mis} and only pay a price of $O(n^{-1/2})$:
 \begin{cor}
 	Consider the setup of Theorem~\ref{thm:mis} and  further assume that $\gams^2$ is bounded below, as $d \to \infty$.  Then, for any $t \ge 0$, there is a constant $C_2 = C_2(\pimin,b,t)$, such that for $n  \ge C_2 \gams^{-2}$,
 	 with probability at least $1-5 \numc^2 \exp(-t^2)$, the average misclassification rate of Algorithm~\ref{alg:kern:clust}
 	satisfies
 	\begin{align}\label{eq:Mis:rate:pop}
 	\Misb \le 2 C_1 F(\gams^2,\vbs^2) + \frac{C_3(t)}{\sqrt{n}},
 	\end{align}
 	where $C_3(t) = 18 b^2 t / (\sqrt{\pimin} \gams^2)$, assuming that  $F(\gamst^2,\vbs^2) +\frac{C_3}{\sqrt{n}} \le C_1^{-1}$.
 \end{cor}

The boundedness assumption in Proposition~\ref{prop:concent:Psi:v} holds if either $\Kt_\sigma$ is uniformly bounded on $\reals^d$ (as in the case of the Gaussian kernel), or  $\{P_k\}$ are supported on some bounded manifolds and $\Kt_\sigma$ is continuous. The second assumption is quite reasonable since it assumes the ``true'' signal $\mu_i$ to be bounded whereas the noisy observation $x_i$ can still have an unbounded distribution.

 In some cases, one might be able to explicitly compute $\gams^2$ and $\vbs^2$ as the next examples show:
\begin{exa}[Nested spheres with isotropic noise] %
	\label{exa:concent:iso:approx}
	Consider the case where $\{P_k\}$ are uniform distributions on nested spheres in $\reals^d$ of various radii: $r_1,\dots,r_\numc$.
	Recalling the definition of $s$ in~\eqref{eq:Kt:s:def:isotropic}, let
	\begin{align*}
		\rt_k = \frac{r_k}{\tau s}, \quad  \ut_k := s^{-d/2} e^{-\rt_k^2/2}, 
		\quad\text{and}\quad  u_k :=  e^{-(r_k^2+\sigma^2)/2{\tau^2}}
	\end{align*}
	for $k \in [\numc]$. 
	 Let $\theta$ and $\theta'$ be independent variables distributed uniformly on the unit sphere $S^{d-1}$, and set $\psi_d(u) = \ex \exp(u \ip{\theta,\theta'})$. Then, it is not hard to see that
	\begin{align*}
	\Psi_{k\ell}^* 
	= \ex\big[ \Kt_\sigma(r_k \theta, r_\ell \theta')\big] 
	&= \ut_k \ut_\ell\, 
	\psi_d\big( \rt_k \rt_\ell \big), \\
	 (v^*_{k\ell})^2 
	= \var \big[ \Kt_\sigma(r_k \theta, r_\ell \theta')\big] 
	&= \ut_k^2 \ut_\ell^2 \Big[ \psi_d(2 \rt_k \rt_\ell) - \psi_d(\rt_k \rt_\ell)\Big].
	\end{align*}
	Although, $\psi_d$ can be written as a Beta integral, let us consider the case of large $d$ (high-dimensional data) which simplifies the expressions. 
	As $d\to \infty$, both $\rt_k$ and $\ut_k$ stabilize since $s \to 1$ and $s^{-d/2} \to e^{-\sigma^2/2\tau^2}$ (see Example~\ref{exa:Gauss:Gauss}). It follows that $\rt_k \to r_k / \tau$ and  $\ut_k \to u_k$.
	One can also show that $\psi_d(u) \approx \exp(u^2 / 4d)$ for $u \ll d$ (see Section~\ref{sec:details:of:example}). Then, $\Psi_{k\ell} \to u_k u_\ell$ and $v_{k\ell}^2 \to 0$ as $d \to \infty$, assuming that the bandwidth $\tau$ and the radii $\{r_k\}$ remain fixed.

	 The population-level approximation error is bounded (up to constants)~by
	\begin{align}\label{eq:circ:approx:err}
		\frac{\vbs^2}{\gams^2} =  %
		O \left( 
		\frac{C_1(u)}{C_2(u)} \,\frac{(r_k r_\ell)^2}{\tau^4 \,d} \right) = O\Big( \frac1{d}\Big), \quad \text{as}\; d \to \infty,
	\end{align}
	which is vanishing as $d$ gets large. Here,
	\begin{align*}
		C_1(u) = \max_k u_k^4, \quad  C_2(u) = \Big(\sum_t \pi_t u_t^2 \Big) \min_{k \neq \ell} (u_k - u_\ell)^2.
	\end{align*}
	To simplify the numerator, we have  used $\psi(u) / \psi(2u) \approx 1-e^{-3u^2/4d} \approx 3u^2/4d$ as $d \to \infty$.
	Note that the prefactor in~\eqref{eq:circ:approx:err} makes intuitive sense: The bound is controlled by the closest sphere to the origin (having largest $u_k$, hence largest variance)  in the numerator and the two closest spheres in the denominator. 

	Let us now consider the population-level estimation error. 
	As discussed in Example~\ref{exa:sc:isotropic}, the estimation error is bounded up to constants by
	\begin{align*}
		\frac{1}{\gams^2} \frac{\sigma^2}{ \tau^2}   \frac1d
		\asymp  \frac{1}{C_2(u)}  
		\frac{\sigma^2}{ \tau^2} \frac1d.  %
	\end{align*}
	Increasing $\tau^2$ decreases the effect of noise by reducing $\sigma^2/\tau^2$, but increases $1/\gams^2 \asymp 1/C_2(u)$ by making $\{u_k\}$ closer, since all $u_k$ approach $1$ as $\tau \to \infty$. This also increases the approximation error~\eqref{eq:circ:approx:err} in general. Thus the bandwidth to noise level $\tau/\sigma$ plays a subtle role in balancing the effect of the two terms.
	Since both the estimation and approximation errors  go down as $O(d^{-1})$,
	KSC is consistent at an overall rate of $O(d^{-1} + n^{-1/2})$, as implied by~\eqref{eq:Mis:rate:pop}.
	\qed
\end{exa}
\begin{exa}[Nested spheres with radial noise]
	\label{exa:concent:aniso:approx} 
	Consider again the nested spheres as the signal model, but this time with (anisotropic) radial noise model of Example~\ref{exa:spheres:anisotrop}. We can proceed as in Example~\ref{exa:concent:iso:approx} in estimating parameters $\gams^2$ and $\vbs^2$. The only difference is that we need to use the appropriate kernel mean matrix $\Kt_\sigma$, given by~\eqref{eq:Kt:def:anisotrop} in this case. Let $u_k = e^{-r_k^2/2\tau^2}$. Then one can show that (cf.~Appendix~\ref{sec:details:aniso:approx}) as $d \to \infty$,
	\begin{align*}
		\Psi_{k\ell}^* 
		= \ex\big[ \Kt_\sigma(r_k \theta, r_\ell \theta')\big] 
		&\to u_k u_\ell,
		\quad
		(v^*_{k\ell})^2 
		= \var \big[ \Kt_\sigma(r_k \theta, r_\ell \theta')\big]
		\asymp \frac{u_k^2 u_\ell^2}{\tau^4 d}.
	\end{align*}
	These estimates are similar to those obtained in Example~\ref{exa:concent:iso:approx}, hence the same bound~\eqref{eq:circ:approx:err} holds for $\vbs^2/\gams^2$ in this case; that is, the population-level approximation error goes down as $O(d^{-1})$, similar to the case of the isotropic noise. 
	Since the estimation error also goes down as $O(d^{-1})$ in this case (cf. Example~\ref{exa:spheres:anisotrop}), 
	 KSC is consistent at an overall rate of $O(d^{-1} + n^{-1/2})$, as implied by~\eqref{eq:Mis:rate:pop}. %
	\qed
\end{exa}

Let us summarize our analysis for the nested spheres example with the isotropic and radial noise models. Assume that $\sigma$, $\tau$ (the kernel bandwidth), $\pimin$ and the radii of the spheres remain constant. For  both noise models, the bound on the approximation error vanishes at a rate $O(d^{-1} + n^{-1/2})$, while the bound on the estimation error vanishes at a rate $O(d^{-1})$, for sufficiently large $n$. Irrespective of which noise   structure one assumes (i.e., radial or isotropic), the KSC is consistent at a rate $O(d^{-1} + n^{-1/2})$ for the nested sphere signal.  This conclusion is corroborated by simulations in Section~\ref{sec:sims}. 

\newcommand\mus{\mu^*}
\subsection{Comparison with existing literature}\label{sec:comparison}
The work of~\cite{yan2016robustness} considers a model of the form~\eqref{eq:nonparam:mix:noise:2} with $P_k = \delta_{\mus_k}$, i.e., point masses at $\{\mus_1,\dots,\mus_\numc\}$, $\sigma=1$ and $\Sigma(\mu^*_k) = \sigma_k^2 I_d$ for $k=1,\dots,\numc$. They consider clustering based on a kernel of the form $K(x,y) = f(\norm{x-y}^2)$ where $f$ is both bounded and Lipschitz. They analyze a Laplacian-based spectral clustering algorithm, using a row and column normalized version of the kernel matrix $K$, and obtain bounds on its misclassification rate, involving the eigenvalues of a block-constant version of $K$. When all the noise variances, and pairwise distances among $\{\mus_k\}$, are equal, the eigenvalue bound simplifies to give a consistency rate of $O(\log d / d)$.

When $d \ge 2$, the class of multivariate Lipschitz kernels allowed by Theorem~\ref{thm:mis} is much larger than that of the bounded distance-based kernels considered in~\cite{yan2016robustness}. For example, a kernel that takes two input images and processes them through a ReLU neural network, with operator-norm bounded weight matrices, falls within the Lipschitz class we consider. We note, however, that the particular form considered in~\cite{yan2016robustness} is not necessarily a Lipschitz kernel unless $\sup_t |t f'(t^2)| < \infty$, hence could fall outside our class. Our result also allows for more general noise and signal structures. In particular, the signal $P_k = \delta_{\mu^*_k}$ considered in~\cite{yan2016robustness}  corresponds to the classical parametric mixtures. This model gives $\vb^2 = 0$ in our result, leading to zero approximation error, hence a $O(1/d)$ convergence rate from Theorem~\ref{thm:mis}, improving the rate of~\cite{yan2016robustness} by a $\log d$ factor for Lipschitz kernels. This holds even if the entire covariance matrix of the noise changes at every data point, as long as $\max_i \opnorm{\Sigma_i} \lesssim 1$. It is also worth noting that, in contrast to bounds based on eigenvalues which are often hard to interpret, our bound is directly in terms of interpretable quantities $\vb^2$ and $\gamma^2$.  On the other hand,  \cite{yan2016robustness} allows for the existence of outliers which we do not consider. They also obtain a strong consistency (exact recovery) result for a semidefinite programming variant of kernel clustering, which falls outside the scope of this paper.


	
	The work of~\cite{schiebinger2015geometry} considers a finite nonparametric mixture model on a \emph{compact} space. Their model is equivalent to~\eqref{eq:nonparam:mix:noise} with the noise component set to zero ($\sigma = 0$), i.e, assuming $X_i \sim \sum_k \pib_k P_k$ with $P_k$ compactly supported. In contrast, we assume $X_i \sim \sum_k \pib_k P_k * N(0, \sigma^2 \Sigma_0 / d)$, in the special case of  constant covariance noise. Here, $*$ denotes convolution. In fact, we can allow for the convolution with any member of the LC class defined earlier, including strongly log-concave densities. This allows us to model mixture components with infinite support on $\reals^d$, a more realistic setup not covered in~\cite{schiebinger2015geometry}. In addition, compactness together with the continuity of the kernel function, assumed in~\cite{schiebinger2015geometry}, implies a bounded kernel while we allow for unbounded Lipschitz kernels. The main focus of~\cite{schiebinger2015geometry} is to establish a geometric property for the embedding of the data points obtained from a Laplacian-based kernel representation.
	
	 Under suitable conditions, \cite{schiebinger2015geometry} establishes what they call an $(\alpha,\theta)$-orthogonal cone structure (OCS) for that embedding~\cite[Theorem~2]{schiebinger2015geometry}. This means that a $1-\alpha$ fraction of the points from each mixture component lie within a cone of angle $\alpha$ centered at one of the coordinate axes. They also show that under further assumptions on $\alpha$ and $\theta$, a randomized kmeans algorithm applied to an embedding, with an $(\alpha,\theta)$-OCS structure, leads to a misclassification rate at most $\alpha$~\cite[Proposition~1]{schiebinger2015geometry}. The implicit nature of  the multiple conditions on $\alpha$ and $\theta$ in these two results, however, makes it difficult to parse out an explicit misclassification rate. Moreover, $\alpha$ at best is a constant and cannot go to zero to establish consistency. In contrast, we provide an explicit misclassification bound in terms of easily computable quantities and derive explicit rates of convergence as $d$ and $n$ diverge.
	
	It is worth noting that our results apply to model~\eqref{eq:nonparam:mix:noise:2} which in its general form (with variable covariance structure) goes beyond even a finite nonparametric mixture model for $\{X_i\}$.  As far as we know, the general case of model~\eqref{eq:nonparam:mix:noise:2} has not been analyzed for clustering before.
	%
The special case in model~\eqref{eq:nonparam:mix:noise} is the same as the signal plus noise model of~\cite{el2010information} with  covariance  matrix $\Sigma$ in that paper replaced with $\sigma \Sigma_0$.
	%
	In contrast to~\eqref{eq:nonparam:mix:noise}, \cite{el2010information} does not consider any structure for the signal and the problem there is only to establish the closeness of the kernel matrix based on the pure signal and that based on the contaminated signal.
	
	Finally, our work is based on the technical machinery developed in~\cite{zhou2019analysis} for the analysis of network spectral clustering. In particular, we leveraged the approach of~\cite{zhou2019analysis} in deriving eigenvalue-free bounds on misclassification rate. The results of~\cite{zhou2019analysis}, however, are not directly applicable to kernel clustering, since the (symmetric) deviation matrix $A - \ex A$ there, is assumed to have independent entries on and above the diagonal. In contrast, the deviation $K - \ex K$ for a kernel matrix does not have independent entries on and above the diagonal. Deriving a concentration bound for such a matrix was the main focus of this paper, allowing us to provide the main missing ingredient of the analysis.
\subsection{Simulations}\label{sec:sims}
We now provide some simulations to corroborate the theory we developed for the kernel spectral clustering. 
We use the ``nested spheres'' example that we analyzed in Sections~\ref{sec:finite-sample:mis} and~\ref{sec:pop:params}. %
We compare the performance of the kernel spectral clustering described in Algorithm~\ref{alg:kern:clust} %
with the Lloyd's algorithm (with \texttt{kmeans++} initialization) applied directly to the data points. 

For the kernel function, we consider the Gaussian kernel with bandwidth set as $\tau^2 = \alpha (1+\sigma^2)$, for $\alpha = 1,2$.
 This scaling of $\tau^2$ in terms of $\sigma^2$ is motivated by the concentration bounds, where the estimation error is controlled by $\sigma^2/\tau^2$.
Constant $1$ is added to avoid degeneracy when $\sigma \to 0$.

In addition to the Gaussian kernel, we also use the simple \emph{pairwise distance} (\texttt{pairDist}) kernel  $K(x,y) = \norm{x-y}$. Since this kernel is $1$-Lipschitz, all the theory developed in the paper  applies in this case, with appropriate modifications to the mean kernel  $\Kt_\sigma$. In particular, one can argue as in Examples~\ref{exa:spheres:anisotrop} and~\ref{exa:concent:aniso:approx} that for the radial noise model, this kernel is also consistent as $n, d \to \infty$. Note that although a more appropriate choice would be $(x,y) \mapsto -\norm{x-y}$ to make the kernel a similarity measure, the sign is irrelevant in spectral clustering.

\begin{figure}[t]
	\centering
	\includegraphics[width=.49\textwidth]{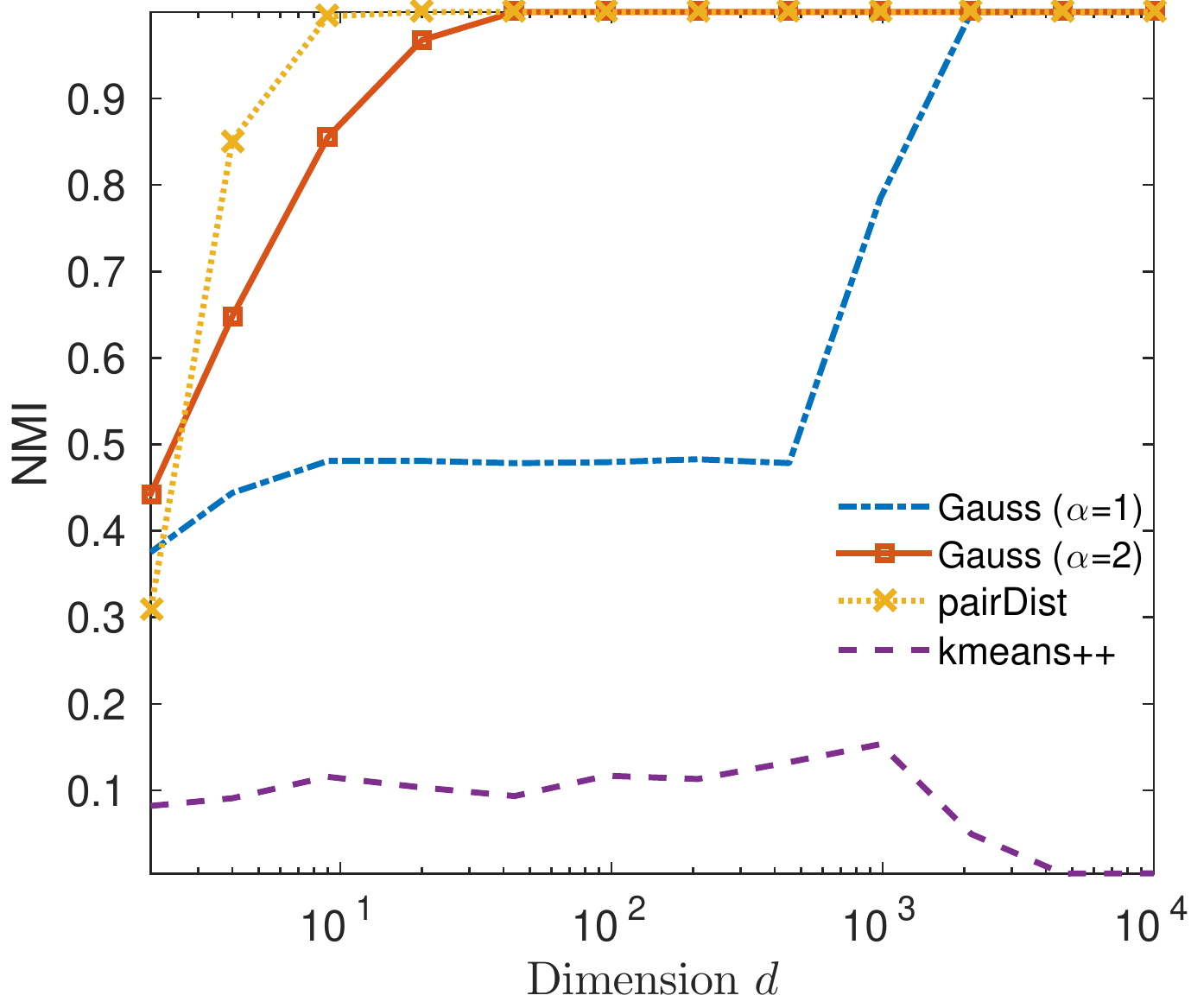}
	\includegraphics[width=.49\textwidth]{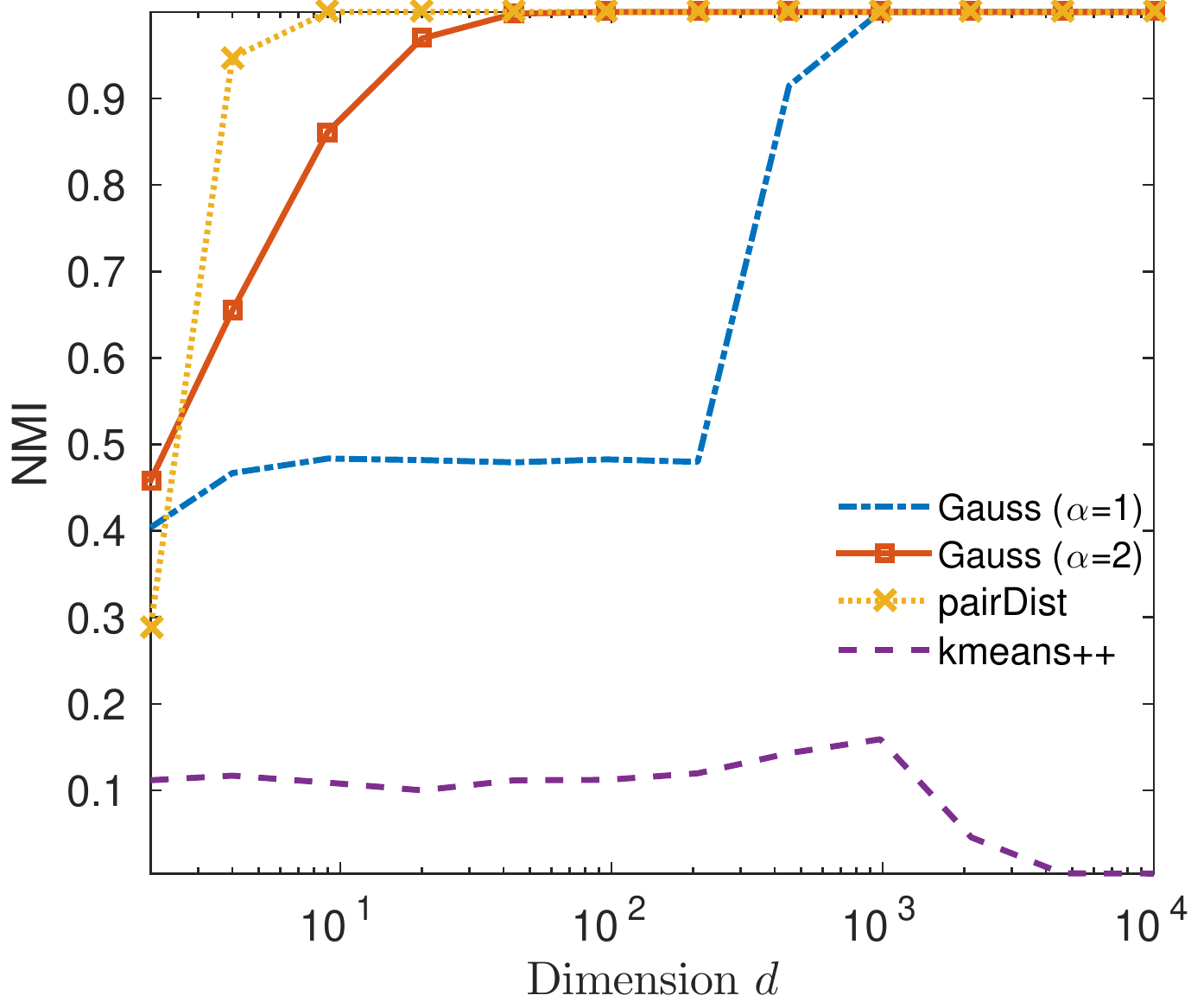}
	\caption{Plots of NMI versus dimension for kernel spectral clustering Algorithm~\ref{alg:kern:clust}, under the noisy ``nested spheres'' model with radii  $r_i = 1,5,10$ (three clusters). Left and right plots correspond to isotropic versus radial noise, respectively. Here, $n=500$, $\sigma=1.5$ and $\tau^2 = \alpha (1+\sigma^2)$.
}
	\label{fig:dvar}
\end{figure}
Figure~\ref{fig:dvar} shows the results. The plots show the normalized mutual information (NMI) versus  dimension $d$, for a fixed value of $\sigma = 1.5$ and a sample size of $n=500$. 
The ``nested spheres'' signal with radii $r_i=1,5,10$ (three clusters) is considered along with both the isotropic and radial noise models.
The plots show the normalized mutual information (NMI) obtained by the KSC algorithm (relative to the true labels) as the dimension varies from $d=2$ to $d=10^4$.
The NMI is a similarity measure between two cluster assignments, more aggressive than the average accuracy. A random clustering against the truth produces NMI $\approx 0$, while a prefect match gives NMI = 1. The plots are obtained by averaging over 12 independent replicates. %

The right and left panels in Figure~\ref{fig:dvar} correspond to the radial and isotropic noise model, respectively. The plots show that, with either noise structure, the KSC Algorithm~\ref{alg:kern:clust} is consistent for the pairwise distance as well as the Gaussian kernel 
for both values of $\alpha$, eventually, as $d$ grows.
These results are as predicted by the theory. 
Note that for the Gaussian kernel with $\alpha=2$, consistency in the isotropic case is achieved at a ``slightly higher dimension $d$'', consistent with the intuition that the isotropic model corresponds to the radial case with spheres ``slightly closer''. The tuition is based on translating the isotropic model to the radial model by projecting the noise onto the sphere. However, the linear projection is  imperfect in putting the transverse noise component exactly on the sphere, hence causing the spheres to appear closer relative to the purely radial noise.

\section{Proofs of the main results}\label{sec:proofs}
Let us start by giving high-level ideas of the  proofs. For Theorem~\ref{thm:Lip:ker:gen}, we first show that $\opnorm{K - \ex K}$ is a Lipschitz function of $X$. The distributions in class LC have the property that any Lipschitz function of $X$ concentrates around its mean. In particular, we obtain that $\opnorm{K - \ex K}$ is concentrated near $\ex \opnorm{K - \ex K}$. We bound this latter expectation by $\ex \fnorm{K - \ex K}$ which in turn is bounded by controlling $\var(K(X_i,X_j))$ for all pairs $(i,j)$, again using the Lipschitiz concentration property.

For Theorem~\ref{thm:ip:ker}, we first derive a tail bound for $|z^T (K - \ex K) z|$, given a fixed $z \in \sph{n}$. This bound requires an extension of the Hanson--Wright  inequality to non-centered variables, which is presented and proved in Appendix~\ref{sec:HW:gen1}. Equipped with the tail bound, we use a discritization argument to obtain uniform control over $\sph{n}$ and complete the proof. 

For Theorem~\ref{thm:mis}, we first approximate the normalized kernel matrix $A = K/n$, in operator norm, by a block-constant matrix, denoted as $\Ks_\sigma / n$. Next, we argue that the eigenvalue-truncated version of $A$, namely $A^{(R)}$, is close to $\Ks_\sigma/n$ in Frobenious norm. Finally, we use $k$-means perturbation results to show that the misclassication error is bounded, up to constants, by $\fnorm{A^{(R)} - \Ks_\sigma/n}^2 / \gamma^2$ where $\gamma^2/n$ is related to the minimum center separation among the rows of $\Ks_\sigma/n$. Combining these bounds leads the desired inquality~\eqref{eq:mis:bound}.

In the rest of this section, we give details of the proofs, starting with some preliminary concentration results.

\subsection{Preliminaries}
Let us start with the following definition (borrowed from~\cite{Adamczak2015-ba} with modifications):
\begin{defn}\label{defn:concent:property}
	A random vector $Z \in \reals^d$ satisfies the  \emph{concentration property} with constant $\kappa > 0$ if
	for any Lipschitz function $f: \reals^d \to \reals$, with respect to the $\ell_2$ norm, we have
	\begin{align}\label{eq:conent:property}
	\pr \Big( f(Z) - \ex f(Z) > t \,\lipnorm{f} \Big) \le \exp({-} \kappa \,t^2 ), \quad\forall t > 0.
	\end{align}
\end{defn}
Note that it is enough to have~\eqref{defn:concent:property} for $1$-Lipschitz functions (i.e., $\lipnorm{f} = 1$) which then implies the general case by rescaling.
The following result is well-known~\cite{ledoux2001concentration}; see also~\cite[Theorem~5.2.2]{vershynin2018high}:
\begin{thm}[]\label{thm:Gauss:concent}
	A standard Gaussian random vector $Z \sim N(0,I_d)$ satisfies the  concentration property with constant $\kappa = 1/2$.
\end{thm}

A similar result holds for a strongly log-concave random vector~\cite[Theorem~5.2.15]{vershynin2018high}:
\begin{thm}\label{thm:log:concave:concent}
	A strongly log-concave random vector $Z \in \reals^d$  with curvature $\alpha^2 > 0$
	satisfies the concentration property with constant $\kappa = C \alpha^2$ for some universal constant $C > 0$.
\end{thm}
This result can be easily extended to a collection of independent strongly log-concave random vectors:
\begin{cor}\label{cor:log:concave:collect}
	Let $Z_1,\dots,Z_n \in \reals^d$ be independent strongly log-concave random vectors with curvatures $\alpha_i^2 > 0$. Then $\vec Z \in \reals^{nd}$ obtained by concatenating $Z_1,\dots,Z_n$ is strongly log-concave with curvature $\alpha^2 := \min_i \alpha_i^2$. In particular, $\vec Z$ satisfies the concentration property with constant $\kappa = C \alpha^2$.
\end{cor}
\begin{proof}
	It is enough to note that $\vec Z$ has density $f(z) = \prod_i e^{-U_i(z_i)} = e^{-U(z)}$ where $U(z) := \sum_i U_i(z_i)$ whose Hessian is block-diagonal with diagonal blocks $\nabla^2 U_i(z_i) \succeq \alpha_i^2 I_d$.
\end{proof}

 We write $\sph{n} = \{x \in \reals^n: \norm{x}_2 = 1\}$ for the sphere in $\reals^n$. We  frequently use the following vector and matrix notations: For $X_1,\dots,X_n \in \reals^d$, we write $X = [X_1 \mid \cdots \mid X_n]$ for the $d \times n$ matrix with columns $X_i$, and let
\begin{align}\label{eq:vec:def}
	X \mapsto \vec X : \reals^{d \times n} \to \reals^{d n}
\end{align}
 be the operator that maps a matrix $X$ to a vector $\vec X$ by concatenating its columns.
\subsection{Proof of Theorem~\ref{thm:Lip:ker:gen}}

The key is the following lemma due to M. Rudelson which is proved in Appendix~\ref{sec:lem:Rud}:
\begin{lem}[Rudelson]\label{lem:Rud}
	Assume that $K(X)$ is as in~\eqref{eq:K(X)} and the kernel function is $L$-Lipschitz as in~\eqref{eq:Lip:kern:ineq}. Then, 
	\begin{itemize}
		\item[(a)] $\fnorm{K(X) -K(X')}^2 \le 4n L^2 \fnorm{X-X'}^2$ for any $X,X' \in \reals^{d \times n}$, and
		\item[(b)] for any $a \in \reals$, $X \mapsto \opnorm{K(X) - a}$ is $2 \sqrt{n} L$-Lipschitz w.r.t. the Frobenius norm.%
	\end{itemize}
\end{lem}

Part~(a) of Lemma~\ref{lem:Rud} can be interpreted as showing that the matrix-valued map $X \mapsto K(X) : \reals^{d \times n} \to \reals^{n \times n}$ is ($2\sqrt{n} L$)-Lipschitz, assuming that both spaces are equipped with the Frobenius norm. As a consequence of Lemma~\ref{lem:Rud}, we get the following concentration inequality:
\begin{prop}\label{prop:V:concent}
	Let $X_i = \mu_i + \sqrt{\Sigma_i} W_i \in \reals^d, i=1,\dots,n$ be random vectors and set $W = [W_1 \mid   \cdots \mid W_n] \in \reals^{d \times n}$. Assume that the random vector $\vec W \in \reals^{dn}$ satisfies the concentration property~\eqref{eq:conent:property} with constant $\kappa = c /\omega^2 > 0$. Let  $K(X)$ be as defined in~\eqref{eq:K(X)} with a kernel function satisfying~\eqref{eq:Lip:kern:ineq}.
	Then, $V:=\opnorm{K - \ex K}$ is sub-Gaussian, and
	\begin{align*}
	\pr\big( V - \ex V \ge 2 \sqrt{n} L \siginf \omega t \big) \le \exp(- c t^2), \quad t \ge 0,
	\end{align*}
	where $\siginf^2 := \max_i \opnorm{\Sigma_i}$.
\end{prop}

An equivalent (up to constant) statement of this result is 
\begin{align*}
\big\|\, \opnorm{K - \ex K}\, \big\|_{\psi_2} \lesssim \sqrt{n}  L \sigma_\infty \omega.
\end{align*}
where $\sgnorm{\cdot}$ denotes the sub-Gaussian norm.

\begin{proof}
	Set $S_i = \sqrt{\Sigma_i}$ and let $S = \diag(S_1,\dots,S_n)$ be the $dn \times dn$ block diagonal matrix with diagonal blocks $\{S_i\}$.
	Also, let $X,W,\mu \in \reals^{d \times n}$ be the matrices with columns $\{X_i\}$, $\{W_i\}$ and $\{\mu_i\}$, respectively.
	Using vector notation~\eqref{eq:vec:def}, we have $\vec X = \vec \mu + S \vec Z$. With some abuse of notation, we  write $K(\vec X)$ to denote $K(X)$ as defined in~\eqref{eq:K(X)}. Note that, $\norm{\vec W} = \fnorm{W}$, that is, the $\ell_2$ norm of vector $\vec W$ is the same as the Frobenius norm of  matrix $W$.
	
	For any $a \in \reals$, we claim that $\vec W \mapsto F(\vec W) := \opnorm{K(\vec \mu + S \vec W) - a}$ is $(2\sqrt{n} L \siginf)$-Lipschitz w.r.t. the $\ell_2$ norm on $\reals^{d n}$. Indeed, 
	\begin{align*}
	|F(\vec W) - F(\vec W')| %
	&\le 2 \sqrt{n} L \norm{S \vec W - S \vec W'} & \text{(By Lemma~\ref{lem:Rud}(b))} \\
	&\le 2  \sqrt{n} L \opnorm{S} \norm{\vec W - \vec W'}
	\end{align*}
	and $\opnorm{S} = \max_i \opnorm{S_i} = \siginf$, since $\opnorm{S_i}^2 = \opnorm{\Sigma_i}$. The result now follows from~\eqref{eq:conent:property} after replacing $t$ with $\omega t$. %
\end{proof}
Next, we bound the expectation of $\opnorm{K - \ex K}$. Here, we pass to the Frobenius norm, giving us an upper bound on the expectation: %

\begin{prop}\label{prop:E:frob:dev:Lip}
	Assume that $\{X_i\}_{i=1}^n$ satisfy the assumption of Proposition~\ref{prop:V:concent}, and let $K = K(X)$ be as defined in~\eqref{eq:K(X)} and satisfies~\eqref{eq:Lip:kern:ineq}. Then, with $C = 2/\sqrt{c}$,
	\begin{align*}
	\ex \fnorm{K - \ex K} \le  C n L \omega \siginf.
	\end{align*}
\end{prop}

\begin{proof}
	By Lemma~\ref{lem:kern:var:bound} below,
	\begin{align*}
	 \ex \fnorm{K - \ex K}^2 &= \sum_{i,j=1}^n \var\big(K(X_i,X_j) \big)  \\
	&\le n^2 \max_{i,j}  \var\big(K(X_i,X_j) \big)  
	\le  C^2 n^2 L^2 \omega^2 \siginf^2.
	\end{align*}
	Noting that $\ex \fnorm{K - \ex K} \le (\ex \fnorm{K - \ex K}^2)^{1/2}$ finishes the proof.
\end{proof}

\begin{lem}\label{lem:kern:var:bound}
	Assume that $X_i = \mu_i + \sqrt{\Sigma_i} W_i \in \reals^d$ are independent for $i=1,2$, and $\vec W = (W_1,W_2) \in \reals^{2d}$ satisfies the concentration property~\eqref{eq:conent:property} with $\kappa = c/\omega^2 > 0$. Then, with $C^2 = 4/c$, 
	\begin{align*}
	\var \big( K(X_1,X_2)\big) &\le C^2 L^2 \omega^2 \max\{ \opnorm{ \Sigma_1},  \opnorm{ \Sigma_2} \}, \\
	\var \big( K(X_1,X_1)\big) &\le C^2 L^2 \omega^2 \opnorm{ \Sigma_1}.
	\end{align*}
\end{lem}

\begin{proof}
	For $x,y \in \reals^d$, let $\vec z = (x,y)$ and define $\Kt:\reals^{2d} \to \reals$ by $\Kt(\vec z) := K(x,y)$. Note that $\Kt$ is $\sqrt{2} L$-Lipschitz w.r.t. to the $\ell_2$ norm on $\reals^{2d}$, that is, $|\Kt(\vec z) - \Kt(\vec {y})| \le \sqrt{2} L \norm{\vec z - \vec {y}}$ for any $\vec z, \vec {y} \in \reals^{2d}$. Let $\vec \mu = (\mu_1, \mu_2) \in \reals^{2d}$, $\vec W = (W_1,W_2)\in \reals^{2d}$ and $\Sigb = \diag(\Sigma_1, \Sigma_2) \in \reals^{2d \times 2d}$. We have $K(X_1,X_2) = \Kt(\vec \mu + \Sigb^{1/2} \vec W)$. We note that 
	\begin{align}
	\lipnorm{ \vec W \mapsto \Kt(\vec \mu + \Sigb^{1/2} \vec W)} \le  \sqrt{2} L \opnorm{ \Sigb^{1/2}} = \sqrt{2} L \siginf^{(12)}
	\end{align}
	where $\siginf^{(12)}:= \opnorm{ \Sigb^{1/2}}  = \max\{ \opnorm{ \Sigma_1^{1/2}},  \opnorm{ \Sigma_2^{1/2}} \}$.
	From the concentration property, it follows that
	\begin{align*}
	\pr \Big( |K(X_1,X_2) - \ex K(X_1,X_2)| > t \, \sqrt{2} L \siginf^{(12)} \omega \Big) 
	\le 2 \exp({-} c \,t^2 ), \quad\forall t > 0.
	\end{align*}
	Letting $\Delta =  K(X_1,X_2) - \ex K(X_1,X_2) $ and $\alpha = \sqrt{2} L \siginf^{(12)} \omega$, we have
	\begin{align*}
	\ex \Delta^2  = \int_0^\infty 2t \pr (|\Delta| > t) dt = 2 \alpha^2  \int_0^\infty t\pr( |\Delta| > \alpha t) dt \le 4 \alpha^2 \int_0^\infty t e^{-ct^2}dt = \frac{2}{c} \alpha^2,
	\end{align*}
	which gives the desired result for $\var(K(X_1,X_2))$ with $C^2 = 4/c$.
	
	For the second assertion, let $J := \big[ \begin{smallmatrix} I_d \\ I_d \end{smallmatrix}\big]$ and note that $K(X_1,X_1) = \Kt(J \mu_1 + J\Sigma_1^{1/2} W_1)$. We also have  $\lipnorm{ W_1 \mapsto \Kt(J \mu_1 + J\Sigma_1^{1/2} W_1)} \le  \sqrt{2} L \opnorm{ \Sigma_1^{1/2}}$. The rest of the argument follows as in the case of $K(X_1,X_1)$.
\end{proof}

Combining Propositions~\ref{prop:V:concent} and~\ref{prop:E:frob:dev:Lip} and noting that $\ex V \le \ex \fnorm{K - \ex K}$ establishes the result for any collection of $\{W_i\}$ for which the concentration property holds for  $\vec W$ with constant $c/\omega^2$. It remains to verify that each case in Definition~\ref{defn:LC} has this property.

\paragraph{Verifying the three cases in the LC class}
 We first deduce the result for part~(b) from~(a). Fix $i$ and $j$ and let $f:\reals \to \reals$ denote the density of $W_{ij}$ w.r.t. the Lebesgue measure, $S$ the support of the distribution, and $F$ the corresponding CDF, i.e., $F(t) = \int_{-\infty}^t f(x)dx$. Pick $x \in S$  and note that $x$ does not belong to flat parts of $F$. Then, by assumption $f(x) \ge 1/\omega$. Let $\nu = F(x)$ so that $x = F^{-1}(\nu)$. By the inverse function theorem, $Q := F^{-1}$ is differentiable at $\nu$ and we have $Q'(\nu) = 1/f(x) \le \omega$. Thus, $Q$ is $\omega$-Lipschitz on $S$. The range of $Q$ restricted to $S$ is $[0,1]$.
 
 Let $\Phi$ be the CDF of the standard normal distribution which is $(1/\sqrt{2\pi})$-Lipschitz. If $Z_{ij} \sim N(0,1)$, then $U_{ij} := \Phi(Z_{ij})$ are uniformly distributed on $[0,1]$ and $Q(U_{ij})$ has the same distribution as $W_{ij}$. In other words, we can redefine $W_{ij} = \phi_{ij}(Z_{ij})$ for $\phi_{ij} = Q \circ \Phi$. We have $\lipnorm{\phi_{ij}} \le \lipnorm{Q} \lipnorm{\Phi} \le \omega /\sqrt{2\pi}$, and the problem is reduced to part~(a), up to constants.

\smallskip
 For part~(a),  we have $W_i = (W_{ij})$ with $W_{ij} = \phi_{ij}(Z_{ij})$ where $Z_{ij} \sim N(0,1)$ are independent across $i=1,\dots,n$ and $j=1,\dots,d$. We define $\vec W$ and $\vec Z$ based on the $d\times n$ matrices $W$ and $Z$ as in~\eqref{eq:vec:def} and compactly write $\vec W = \phi( \vec Z)$. Let $f :\reals^{d n} \to \reals$ be a $1$-Lipschitz function and define $g(\vec Z) := f(\phi(\vec Z)) = f(\vec W)$. Then,
 \begin{align*}
 \norm{g(\vec Z)  - g(\vec Z')}^2 &\le \sum_{ij} \big(\phi_{ij} (Z_{ij}) - \phi_{ij}( Z_{ij}')\big)^2  \\
 &\le \sum_{ij} \lipnorm{\phi_{ij}}^2 (Z'_{ij} - Z_{ij})^2 \le \omega^2 \norm{\vec Z' -  \vec Z}^2,
\end{align*}
 for any vectors $\vec Z, \vec Z' \in \reals^{dn}$. It follows that $g$ is $\omega$-Lipschitz, hence by the concentration of Gaussian measure (Theorem~\ref{thm:Gauss:concent}), we have
 \begin{align*}
 	\pr\big( g(\vec Z) - \ex g(\vec Z)  \ge \omega t \big) \le \exp(-t^2/2).
 \end{align*}
 Since $g(\vec Z) = f(\vec W)$, we have the concentration property for $\vec W$ with constant $1/(2 \omega^2)$.

 \smallskip
 For part~(c), since each $W_i$ has a strongly log-concave density with curvature $\curvature_{i}^2 \ge 1/\omega^2$, it follows from Corollary~\ref{cor:log:concave:collect} that $\vec W$ is strongly log-concave with curvature $1/\omega^2$. Then, by Theorem~\ref{thm:log:concave:concent}, $\vec W$ satisfies the desired concentration property with constant $C / \omega^2$.

\subsection{Proof of Proposition~\ref{prop:lip:lower:bound}}
	Let us define $\phi:\reals \to \reals$ by setting $\phi(x)$  equal to $-\sqrt{L\sigma}$, $x\sqrt{L/\sigma}$ and $\sqrt{L\sigma}$ on $[-\infty,-\sigma]$, $[-\sigma,\sigma]$ and $[\sigma,\infty)$. Let $K(x,y) := \phi(x)\phi(y)$. We note that $\phi$ is $\sqrt{L\sigma}$-bounded and $\sqrt{L/\sigma}$-Lipschtiz,   hence $K(\cdot,\cdot)$ is $L$-Lipschtiz.
	 Let $u_i = \phi(X_i)$ and $u = (u_i) \in \reals^n$. We have $\ex K(X) = \alpha I_n$ where $\alpha = \ex[\phi(X_1)]^2 \le L\sigma$, and $K(X) = u u^T$. 
	 
	 Let $Z_i = 1\{|X_i| > \sigma\}$. When $Z_i = 1$, $u_i = \pm \sqrt{L \sigma}$, hence $u_i^2 Z_i = L \sigma Z_i$. Assuming $\norm{u}^2 \ge \alpha$,  we have
	$    \opnorm{K - \ex K} = \norm{u}^2 - \alpha \ge \sum_{i} u_i^2 Z_i - \alpha \ge L\sigma (\sum_{i} Z_i-1)$.
	Since $\sum_i Z_i\sim \text{Bin}(n,\frac12)$, by the Hoeffding's inequaltiy, $\pr(\sum_i Z_i \le n/4) \le \exp(-n/8)$. On the complement of this event, $\sum_{i} Z_i-1 \ge n/8$ when $n\ge 8$, finishing the proof.

\subsection{Proof of Theorem~\ref{thm:ip:ker}}\label{sec:proof:thm:ip:Ker}
We can write $K = X^T X$ where $X = (X_1 \mid \cdots \mid X_n) \in \reals^{d \times n}$ has $\{X_i\}$ as its columns. Let us fix $z \in \sph{n}$ and consider 
\begin{align}\label{eq:Yz:def}
Y_z := z^T(K - \ex K) z = \norm{X z}^2 - \ex \norm{Xz}^2.
\end{align}
Let $\Xt_i = X_i - \mu_i$ be the centered version of $X_i$, and let $\Xt \in \reals^{d \times n}$ be the matrix with columns $\{\Xt_i\}$. Setting $\mu_z = Mz =\sum_i z_i \mu_i$, we have $Xz = \mu_z + \Xt z$, hence
\begin{align*}
Y_z = \norm{\Xt z}^2 - \ex  \norm{\Xt z}^2 + 2\ip{\mu_z, \Xt z} 
\end{align*}
using the fact that $\Xt z$ is zero-mean. %

\begin{lem}\label{lem:Yz:tail:gen}
	For any $z \in \sph{n}$, $Y_z$ in~\eqref{eq:Yz:def} based on $X_i = \mu_i + \sqrt{\Sigma_i} W_i$ is sub-exponential and
	\begin{align}\label{eq:Yz:tail:gen}
	\pr \bigl(|Y_z| \ge \kappa^2 \siginf^2 t \bigr) \le 4 \exp \Big[ {-c \min  \Big( \frac{t^2}{   d +  \kappa^{-2} \siginf^{-2} \opnorm{M}^2}, t\Big)}\Big].
	\end{align}
\end{lem}
Recalling $\eta =  d +  \kappa^{-2} \siginf^{-2} \opnorm{M}^2$, and changing $t$ to $\eta t$,  \eqref{eq:Yz:tail:gen} can be written as 
\[
\pr\big( |Y_z| \ge  \kappa^2 \siginf^2 \eta \, t \big) \;\le\; 4 \exp \big[ {-c \,\eta \min  \big( t^2, t \big)}\big].
\]
Letting $\delta = (\sqrt{C  n} + u)/\sqrt{\eta}$ and setting $t = \max(\delta^2, \delta)$, we obtain
\begin{align*}
\pr \big(|Y_z| \ge \kappa^2  \siginf^2 \eta \, \max(\delta^2,\delta) \big) \;\le\; 4 \exp(- c \,\eta \,\delta^2) \;\le\; 4 \exp[-c (C n + u^2)].
\end{align*}

We can now use a discretization argument. %
Let $\Nnet$ be a $\frac14$-net of $\sph{n}$, so that $|\Nnet| \le 9^n$. We have
$\opnorm{K - \ex K} = \sup_{z\, \in\, \sph{n}} |Y_z| \le 2 \max_{z\, \in\,\Nnet} |Y_z|$; see for example~\cite[Exercise~4.4.3]{vershynin2018high}.
Letting $\eps = 2 \kappa^2  \siginf^2 \eta \, \max(\delta^2,\delta)$, we have
\begin{align*}
\pr( \opnorm{K - \ex K} \ge \eps ) &\le \pr \big(\max_{z \in \Nnet} | Y_z|  \ge \eps/2\big) \\
&\le 4 \cdot 9^n \exp[-c (C n + u^2)] \;\le\; 4 \exp[-c (C_1  n + u^2)]
\end{align*}
where $C_1 = C - \log 9/c$ which can be made positive by take $C > \log 9/c$.

\begin{proof}[Proof of Lemma~\ref{lem:Yz:tail:gen}]
	Without loss of generality, assume $\Sigma_i \succ 0$ for all $i$.
	Define $Y_z$ as in~\eqref{eq:Yz:def} based on  $X_i = \mu_i + \sqrt{\Sigma_i} W_i$.
	Using  vector notation~\eqref{eq:vec:def}, we have $\vec X = \vec \mu + \sqrt{\Sigma} \vec W$ where $\Sigma = \diag(\Sigma_1,\dots,\Sigma_n)$ is the $nd \times nd$ block diagonal matrix with diagonal blocks $\Sigma_i$. We have $z^T K z = \norm{\sum_i z_i X_i}^2 = \norm{X z}^2$.  Let $\Gamma_z = z^T \kron I_d \in \reals^{d \times nd}$ where $\kron$ is the Kronecker matrix product. We note that
	\begin{align}\label{eq:kron:vec:identity}
	\Gamma_z \vec X = (z^T \kron I_d) \vec X =
	\begin{bmatrix}
	z_1 I_d &z_2 I_d& \cdots & z_n I_d
	\end{bmatrix} 
	\begin{bmatrix}
	X_1 \\ \vdots \\X_n
	\end{bmatrix} =
	X z.
	\end{align}
	It follows that
	\begin{align*}
	Xz = \Gamma_z \vec \mu + \Gamma_z \sqrt{\Sigma} \vec W = \Gamma_z \sqrt{\Sigma}\bigl(\Sigma^{-1/2} \vec \mu + \vec W\bigr).
	\end{align*}
	Letting $\vec \xi := \Sigma^{-1/2} \vec \mu + \vec W$, we have 
	\begin{align*}
	\norm{Xz}^2 = \norm{ \Gamma_z \sqrt{\Sigma}\, \vec \xi}^2 = \vec \xi^T A_z \vec \xi 
	\end{align*}
	where $A_z := %
	\sqrt{\Sigma}^T \Gamma_z^T \Gamma_z  \sqrt{\Sigma}$. Hence, $Y_z := z^T(K - \ex K) z =  \vec \xi^T A_z \vec \xi  - \ex( \vec \xi^T A_z \vec \xi )$ and we can apply the extension of Hanson--Wright inequality, Theorem~\ref{thm:HW:gen1} in Appendix~\ref{sec:HW:gen1} (with $d=1$ and $n$ replaced with $nd$), to obtain
	\begin{align*}
	\pr (|Y_z| \ge \kappa^2 t ) \le 4 \exp \Big[ {-c \min  \Big( \frac{t^2}{  \fnorm{A_z}^2 +  \kappa^{-2} \fnorm{M A_z}^2}, \frac{t}{ \opnorm{A_z}}\Big)}\Big],
	\end{align*}
	where $M = (\Sigma^{-1/2} \vec \mu)^T \in \reals^{1 \times nd}$. We obtain $M A_z =  \vec \mu^T  \Gamma_z^T \Gamma_z  \sqrt{\Sigma}$. Using the inequality $\fnorm{A B} \le \opnorm{A}\fnorm{B}$ ($*$) which holds for any two matrices $A$ and $B$, we have
	\begin{align*}
	\fnorm{M A_z}^2 =  \norm{    \sqrt{\Sigma}^T \Gamma_z^T \Gamma_z \vec \mu }_2^2 
	\le \norm{\sqrt{\Sigma}}^2 \norm{\Gamma_z}^2 \norm{\Gamma_z \vec \mu}_2^2 \le \siginf^2 \norm{\Gamma_z \vec \mu}_2^2 
	\end{align*}
	since $\norm{\Gamma_z} = \norm{z}_2 \norm{I_d} = 1$ and $\norm{\sqrt{\Sigma}}^2 = \norm{\Sigma} = \max_i \norm{\Sigma_i} = \siginf^2$ where the last equality is by definition. Also, by identity~\eqref{eq:kron:vec:identity}, $\Gamma_z \vec \mu = M z$. Hence, $\sup_{z \in \sph{d}} \norm{\Gamma_z \vec \mu} = \opnorm{M}$. Putting the pieces together, $\fnorm{M A_z}^2 \le \siginf^2 \opnorm{M}^2$.
	
	Now, consider the operator norm of $A_z$, for which we have
	\begin{align*}
	\opnorm{A_z} \le \norm{\sqrt{\Sigma}}^2 \norm{\Gamma_z}^2 = \siginf^2.
	\end{align*}
	Finally, for the Frobenious norm of $A_z$,
	\begin{align*}
	\fnorm{A_z} \le \opnorm{\sqrt{\Sigma}}^2 \norm{\Gamma_z} \fnorm{\Gamma_z} = \siginf^2 \sqrt{d}
	\end{align*}
	by repeated application of  matrix inequality ($*$) and $\fnorm{\Gamma_z}^2 = d \norm{z}_2^2 = d$. We obtain
	\begin{align*}
	\pr (|Y_z| \ge \kappa^2 t ) \le 4 \exp \Big[ {-c \min  \Big( \frac{t^2}{  \siginf^4 d +  \kappa^{-2} \siginf^2 \opnorm{M}^2}, \frac{t}{ \siginf^2}\Big)}\Big].
	\end{align*}
	Changing $t$ to $t \siginf^2$ gives the desired result.
\end{proof}

\subsection{Proof of Theorem~\ref{thm:mis}}%
Consider 
a block-constant approximation of $\Kt(\mu)$, denoted as $\Ks_\sigma \in \reals^{n \times n}$, and defined as follows: 
\begin{align}\label{eq:Ks:def}
[\Ks_\sigma]_{ij} = \Psi_{k\ell}, \quad \text{whenever} \; (i,j) \in \Cc_k \times  \Cc_\ell,
\end{align}
where $\{\Psi_{k\ell}\}$ are the empirical averages defined in~\eqref{eq:Psi:def}.
Let $Z \in \{0,1\}^{n \times K}$ be the membership matrix with rows $z_i^T$. It is not hard to see that $\frac1n \Ks_\sigma = Z (\Psi/n) Z^T$ which resembles the mean matrix of a stochastic block model on the natural sparse scaling (see Eq.~(4) in~\cite{zhou2019analysis}). 

The first step of the proof is to to show that the empirical (normalized) kernel matrix $K(X)/n$ is close of $K^*_\sigma/n$. Let us write 
\begin{align*}
	\sqrt{a} := \frac1n \opnorm{ K(X) - \Ks_\sigma }, \quad 
\sqrt{\omega} := \frac1n \opnorm{ K(X) - \Kt_\sigma(\mu)}, 
\end{align*}
and $\sqrt{b} := \frac1n \opnorm{  \Kt_\sigma(\mu) - \Ks_\sigma}$. Using the definition of $v_{k\ell}$ in~\eqref{eq:Psi:def}, 	
\begin{align*}
b \;\le\; \frac1{n^2} \fnorm{  \Kt_\sigma(\mu) - \Ks_\sigma}^2 
&= 	\frac1{n^2} \sum_{k,\ell} \sum_{i,j} z_{ik} z_{j\ell} 
\big( \Kt_\sigma(\mu_i,\mu_j) - [\Ks_\sigma]_{ij}\big)^2  \\
&=	\frac1{n^2} \sum_{k,\ell} n_k n_\ell v_{k\ell}^2 = \vb^2.
\end{align*}
To control $\omega$, note that $\Kt_\sigma(\mu) = \ex\big[ K(X)\big]$ and apply Theorem~\ref{thm:Lip:ker:gen} 
	with $\Sigma_i = \sigma^2 \Sigma(\mu_i)/d$, $c=1/2$, $C = \sqrt{2}$ and $t$ replaced with $\sqrt 2 t$, to get with probability $\ge 1-e^{-t^2}$, %
	\begin{align*}
	n^2 \omega =  \opnorm{K(X) - \ex K(X)}^2 
	&\le 4 L^2  \siginf^2 \big(\sqrt 2 n + \sqrt{2 n} t \big)^2 \\
	&\le \frac{8 L^2 \sigma^2}{d} \max_i \opnorm{\Sigma(\mu_i)} \big( n + \sqrt{ n} t\big)^2.
	\end{align*}

By triangle inequality, $a \le 2 ( \omega + b )$. Thus, recalling the defintion of $F(\gamma^2,\vb^2)$,
\begin{align}\label{eq:a:bound}
    a \le \frac{\gamma^2}{8R} F(\gamma^2,\vb^2).
\end{align}

Let $A := K(X)/ n$ and $A^{(\numc)}$ be obtained by truncating the EVD of $A$ to its $\numc$ largest eigenvalues in absolute value. The second step is to control the deviation of $A^{(\numc)}$ from the block-constant matrix $\Ks_\sigma/n$. Lemma 6 in~\cite{zhou2019analysis} gives
\begin{align}
    \fnorm{A^{(\numc)} - (\Ks_\sigma/n)}^2 \le 8 \numc\, \opnorm{A - (\Ks_\sigma/n)}^2 = 8R a =:\eps^2.
\end{align}

The third and final step is to apply perturbation results for the $k$-means step of the algorithm. We note that $\Ks_\sigma/n$ is a $k$-means matrix with $\numc$ centers, meaning that it has (at most) $\numc$ distinct rows. Let us refer to these distinct vectors as $q_1,\dots,q_\numc \in \reals^n$. Let $\delta_k$ be the minimum $\ell_2$ distance of $q_k$ from $q_j, j \neq r$. Then, $n \delta_k^2 = \min_{\ell: \ell \neq k
} D_{k\ell}$ where $D_{k\ell}$ is as defined in~\eqref{eq:gam}. Now, Corollary~1 in~\cite{zhou2019analysis} implies that if $ \eps^2 / (n \pi_k \delta_k^2) =  \eps^2 / (n_k \delta_k^2) < [4(1+\kappa)^2]^{-1} = C_1^{-1}$, we have
\begin{align*}
    \Misb \le C_1 \frac{\eps^2}{   \min_k (n\delta_k^2)} = C_1 \frac{8R a}{\gamma^2} \le C_1  F(\gamma^2,\vb^2)
\end{align*}
using the defintion of $\gamma^2$ in~\eqref{eq:gam} and  inequality~\eqref{eq:a:bound}.
Since, by definition, $\gamt^2 = \min_k(n \pi_k \delta_k^2)$, the requirted condition holds if $8\numc a/  \gamt^2 = \eps^2 / \gamt^2 \le C_1^{-1}$. A further sufficient condition, in view of~\eqref{eq:a:bound}, is
\begin{align*}
    F(\gamt^2,\vb^2)=  \frac{\gamma^2 F(\gamma^2,\vb^2)}{\gamt^2} \le C_1^{-1}.
\end{align*}
This  finishes the proof for the case where one runs the $k$-means algorithm on the rows of $A^{(\numc)}$. Since the pairwise distance among the rows of $\Uh_1 \Lamh_1$ is the same as that of $A^{(\numc)}$, and the $k$-means algorthim is assumed isometry-invariant, the same result holds for $\Uh_1 \Lamh_1$. The proof is complete.

\subsection{Proof of Proposition~\ref{prop:concent:Psi:v}}
	Let $Y_1,\dots,Y_n$ be an independent sequence of variables and consider the $U$-statistic $U = \binom{n}{2}^{-1} \sum_{i < j} h(Y_i,Y_j)$ for some symmetric $b$-bounded function $h$. Then, one has the following consequence of bounded difference inequality~\cite[Example~2.23]{wainwright2019high}:
	\begin{align*}
	\pr (|U - \ex U|  > t \sqrt{8b^2/ n} ) \le 2e^{-t^2}.
	\end{align*}
	Applying this result with $Y_i = \mu_i$ for $i \in \Cc_k$ and $h = \Kt_\sigma$, with probability at least  $1-2e^{-t^2}$, 
	\begin{align*}
	| \Psi_{kk} - \Psi^*_{kk}| \le t 	\frac{n_k-1}{n_k} \sqrt{8b^2/n_k} \le t \sqrt{8b^2/n_k}.
	\end{align*}
	Now assume that $Y_1,\dots,Y_n,Z_1,\dots,Z_m$ are independent and let 
	\[
	V = (nm)^{-1} \sum_{i,j} h(Y_i,Z_j).
	\]
	Then, by a similar bounded difference argument,
	\begin{align*}
	\pr (|V - \ex V|  > t \sqrt{8b^2/ \min\{m,n\}} ) \le 2e^{-t^2}.
	\end{align*}
	For $k \neq \ell$, applying this result with $Y_i = \mu_i, i \in \Cc_k$ and $Z_j = \mu_j, j\in \Cc_\ell$ gives the desired result. For the variance, we have $v_{k\ell}^2 = \ex \Kt_\sigma^2(X,Y) - \Psi_{k\ell}^2$ where $(X,Y) \sim \Ph_{k\ell}$. The first term is controlled similarly with $b$ replaced with $b^2$, since $\Kt_\sigma^2$ is $b^2$-bounded. For the second term, assume that $|\Psi_{k\ell} - \Psi_{k\ell}^*|\le \delta_{k\ell}$. Then, $|\Psi_{k\ell}^2 - (\Psi_{k\ell}^*)^2| \le 2 b\, \delta_{k\ell}$. %
	Thus, under the event that the bounds hold, we have
	\begin{align*}
	| v_{k\ell}^2 - (v_{k\ell}^*)^2| \le  \frac{t\sqrt 8 b^2}{\sqrt{n_k \wedge n_\ell}} 
	+ (2 b)\frac{t \sqrt 8 b}{\sqrt{n_k \wedge n_\ell}}.
	\end{align*}
	By a similar argument, $|D_{k\ell} - D_{k\ell}^*| \le 8b\,\delta_{k\ell}$.
	Applying union bound over $2 R^2$ pairs, required for controlling $\Psi_{k\ell}$ and $v_{k\ell}^2$, finishes the proof.

\section*{Acknowledgment}
We thank Mark Rudelson for helpful comments, in particular, for the idea behind Lemma~\ref{lem:Rud}.

\begin{supplement}[id=supp]
	\sname{Supplement}%
	\stitle{Technical lemmas}
	\sdescription{This supplement collects some technical results used in the paper.}
\end{supplement}

\bibliographystyle{unsrt}
\bibliography{kernel_refs} %
 
\newpage
\begin{center}
	\Large Supplement for ``Concentration of kernel matrices with application to kernel spectral clustering''
\end{center}

This supplement contains appendices collecting some technical results used in the paper.

\appendix

\section{Auxiliary results}

\subsection{Distance kernels are Lipschitz}\label{app:dist:kern:lip}

\begin{lem}\label{lem:dist:kern:lip}
	A distance kernel $K$ defined as in~\eqref{eq:dist:kern} is $L$-Lipschitz in the sense of~\eqref{eq:Lip:kern:ineq}.
\end{lem}
\begin{proof}
	Let $K(x_1,x_2) = f(\norm{x_1 - x_2}))$ where $f: \reals \to \reals$ is $L$-Lipschitz.
	Then, $|K(x_1,x_2) - K(y_1,y_2)| \le L | \norm{x_1-x_2} - \norm{y_1-y_2}|$ using the fact that $f$ is $L$-Lipschitz. Inequality~\eqref{eq:Lip:kern:ineq} follows from: $|\norm{a - b} - \norm{c-d}| \le \norm{a-c} + \norm{b-d}$.
\end{proof}

\subsection{Proof of Lemma~\ref{lem:Rud}}\label{sec:lem:Rud}
For part~(a), we write
\begin{align*}
\norm{K(X) - K(X')}_F^2 &= \sum_{ij} \big[K(X_i,X_j) - K(X'_i,X'_j) \big]^2  \\
&\le L^2 \sum_{ij} \big[\norm{X_i - X'_i} + \norm{X_j - X'_j} \big]^2 \\
&\le 2 L^2 \sum_{ij} \big[\norm{X_i - X'_i}^2 + \norm{X_j - X'_j}^2 \big] \\
&= 4 nL^2 \fnorm{X-X'}^2
\end{align*}
where the first inequality follows from~\eqref{eq:Lip:kern:ineq}.	For part~(b), let  $F(X) = \opnorm{K(X) - a}$. Then, $|F(X) - F(X')| \le \opnorm{K(X) - K(X')} \le \fnorm{K(X) - K(X')}$.

\subsection{Hanson--Wright inequality for sub-Gaussian vectors}\label{sec:HW:gen1}

In this appendix, we give a %
a generalization of   Hanson--Wright inequality for the sub-Gaussian chaos~\cite[Section~6.2]{vershynin2018high} which could be of independent interest. 
For a matrix $A = (a_{ij}) \in \reals^{n \times n}$, let us write $A^{\Ss} = (A + A^T)/2$ for the symmetric part of $A$. We have
\begin{align}\label{eq:tr:ASB}
	\tr(A^\Ss B) = \tr(A^\Ss B^T) = \tr(A B^\Ss) , \quad \forall A,B \in \reals^{n \times n}.
\end{align}%

\begin{thm}\label{thm:HW:gen1}
	Let $\{X_i, i=1,\dots,n\} \subset \reals^d$ be a collection of independent random vectors, each with independent sub-Gaussian coordinates. Let $\mu_i = \ex[X_i]$ and
	\begin{align*}
	M = (\mu_1 \mid \cdots \mid \mu_n) \in \reals^{d \times n}, 
		\quad \kappa = \max_{i,k} \sgnorm{X_{ik} - \ex X_{ik}}. %
	\end{align*}
	Let $A = (a_{ij})$ be an $n \times n$ matrix and  $Z = \sum_{ij} a_{ij}  \ip{X_i,X_j}$. Then, for any $t \ge 0$,
	\begin{align}\label{eq:HW:gen1}
	\pr\big( |Z - \ex Z| \ge  \kappa^2 t \big) \;\le\; 4 \exp \Big[ {-c \min  \Big( \frac{t^2}{  d \fnorm{A}^2 +  \kappa^{-2} \fnorm{M A^\Ss}^2}, \frac{t}{ \opnorm{A}}\Big)}\Big].
	\end{align}
\end{thm}
Theorem~\ref{thm:HW:gen1} can be thought of as providing a concentration inequality for a general linear functional of the inner product kernel: $Z = \tr(A^T K) = \ip{A,K}$ where $K = (\ip{X_i,X_j})$.
The original Hanson--Wright inequality for sub-Gaussian variables corresponds to the case $d=1$ and $M =0$. We have used the case $d=1$ ($n$ changed to $nd$) and $M\neq 0$ in proving Lemma~\ref{lem:Yz:tail:gen}. %

\begin{proof}[Proof of Theorem~\ref{thm:HW:gen1}]
	Let us first prove the case where $ M = 0$. Let $Z_k = \sum_{ij} a_{ij} X_{ik} X_{jk}$ so that $Z = \sum_{k=1}^d Z_k$, and note that this is a sum of independent terms. Without loss of generality, assume $\kappa=1$. The proof of the 1-dimensional Hanson-Wright~\cite[Chapter~6]{vershynin2018high} %
	shows that
	\begin{align*}
	\ex e^{\lambda Z_k} \le \exp(\lambda^2 C_1 \fnorm{A}^2), \quad \text{for all}\; |\lambda| \le \frac{1}{C_1 \opnorm{A}}
	\end{align*}
	for some constant $C_1 > 0$. By independence of $\{Z_k\}$, we obtain
	\begin{align*}
	\ex e^{\lambda Z} = \prod_k \ex e^{\lambda Z_k} \le \exp( \lambda^2 C_1 d \fnorm{A}^2) \quad \text{for all}\; |\lambda| \le \frac{1}{C_1 \opnorm{A}}
	\end{align*}
	which combined with Lemma~\ref{eq:sube:mgf:tail} below gives the result.
	
	Now consider the general case, with possibly nonzero $M$. Let $\Xt_i = X_i - \mu_i$ be the centered version of $X_i$, and let $X,\Xt$ and $M$ be the matrices with columns $\{X_i\}$, $\{\Xt_i\}$ and $\{\mu_i\}$, respectively. First note that $Z = \tr(A^T X^T X) = \tr(A^\Ss X^T X)$ using~\eqref{eq:tr:ASB} with $B = X^TX$. Let $Y = \tr(A^\Ss \Xt^T \Xt)$ and $R =  M A^\Ss$. Then, we have
	\begin{align*}
		Z - \ex Z &= Y - \ex Y + 2 \tr(R^T \Xt).
	\end{align*}
	 We can apply the zero-mean version of the result to the deviation $Y-\ex Y$. For the second term, we note that $\tr(R^T \Xt) = \sum_{ik} R_{ik} \Xt_{ik}$ which is a sum of independent sub-Gaussian variables, hence 
	 $
	 	\sgnorm{\tr(R^T \Xt)}^2 \lesssim \sum_{ik} R_{ik}^2 \sgnorm{\Xt_{ik}}^2 \le \kappa^2 \fnorm{R}^2
	 $, giving the tail bound
	 \begin{align*}
	 	\pr \big( |\tr(R^T \Xt)| > \kappa \fnorm{R} t \big) \;\le\; 2 \exp (- c t^2), \quad \forall t \ge 0.
	 \end{align*}
	 Combining we have
	 \begin{align*}
	 	\pr \big(|Z - \ex Z| \ge 2 \kappa^2 t \big) & \le \pr \big(|Y - \ex Y| \ge \kappa^2 t \big) + \pr( |\tr(R^T \Xt)| \ge \kappa^2 t) \\
	 	&\le 2 \exp \Big[ {-c \min  \Big( \frac{t^2}{  d \fnorm{A}^2}, \frac{t}{ \opnorm{A}}\Big)}\Big] + 
	 	2 \exp \Big[ {-c}\frac{t^2}{  \kappa^{-2} \fnorm{R}^2}\Big] \\
	 	&\le 4 \exp \Big[ {-c \min  \Big( \frac{t^2}{  d \fnorm{A}^2 +  \kappa^{-2} \fnorm{R}^2}, \frac{t}{ \opnorm{A}}\Big)}\Big] 
	 \end{align*}
	 which is the desired result.
\end{proof}

We recall the following sub-exponential concentration result used in the proof of Theorem~\ref{thm:HW:gen1}:
\begin{lem}\label{eq:sube:mgf:tail}
	Assume that $X$ is a zero-mean random variable satisfying $\ex e^{\lambda X} \le \exp(\lambda^2 v^2 /2)$ for $|\lambda| \le 1/\alpha$. Then,
	$\pr(|X| \ge t) \le 2 \exp({-\frac12 \min\{\frac{t^2}{v^2}, \frac{t}{\alpha}\}})$ for all $t \ge 0$.
\end{lem}

\section{Details of examples}\label{sec:details:exa}
\subsection{Details of Example~\ref{exa:spheres:anisotrop}}\label{sec:details:Ker:aniso}
Consider the eigen-decomposition of 
\begin{align}\label{eq:eigndecomp}
\ut\ut^T + \vt \vt^T = \lambda_1 x_1x_1^T + \lambda_2 x_2 x_2^T
\end{align}
where $\{x_1,x_2,\dots,x_d\}$ is an orthonormal basis. Let us write $u_i = \ip{u,x_i}, i=1,2,\dots,d$ for the components of $u$ along this basis and similarly for $v_i = \ip{v,x_i}$ and $w_i = \ip{w,x_i}$. Note that $w_i = 0$ for $i > 2$, almost surely. Similarly, $u_i = v_i = 0$ for $i > 2$. We also have $w_i \sim N(0,\lambda_i)$ for $i=1,2$ and the two coordinates are independent. It follows that %
\begin{align*}
\Kt_\sigma (u,v) 
&=\ex \exp\Big[ {-}\frac{1}{2\tau^2} 
\sum_{i=1}^2 \big(u_i - v_i + \frac{\sigma}{\sqrt{d}} w_i\big)^2 \Big] \\
&=\frac{1}{s_1 s_2}\exp\Big[ {-}\frac1{2\tau^2} \sum_{i=1}^2 \frac{(u_i-v_i)^2}{s_i^2}\Big], \quad s_i^2 = 1+ \frac{\sigma^2 \lambda_i}{\tau^2 d}
\end{align*}
using Lemma~\ref{lem:Kt:1d} in Appendix~\ref{app:Gauss:ker:comp}.
Let $\ut_i = \ip{\ut,x_i}$, $\vt_i = \ip{\vt,x_i}$ and $\alpha = \ip{\ut,\vt}$. Assuming that $\lambda_1 \ge \lambda_2$, it is not hard to see that $\lambda_1 = 1+|\alpha|$ and $\lambda_2 = 1-|\alpha|$. We also have $\ut_1^2 = \vt_1^2 = \frac12(1+|\alpha|)$ and $\ut_2^2 = \vt_2^2 = \frac12(1-|\alpha|)$ (which can be obtained by multiplying~\eqref{eq:eigndecomp} by $\ut^T$ and $\ut$, and solving the resulting system, and similarly for $\vt$.) This system has eight solutions out of which we have to pick four (the two eigenvectors up to their sign ambiguities).

We have $\lambda_i x_i = \ut_i \ut + \vt_i \vt$ and applying the eigenvector definition, we obtain (assuming that $\{\ut,\vt\}$ are linearly independent) $ (1-\lambda_i) u_i = - \alpha v_i$. When $\alpha \neq 0$, we obtain $\vt_1 = \sign(\alpha) \ut_1$ and $\vt_2 = -\sign(\alpha) \ut_2$.
The result in~\eqref{eq:Kt:def:anisotrop} follows by noting that $u_1 = \norm{u} \ut_1$ and $v_1 = \norm{v} \vt_1$.

\subsection{Details of Example~\ref{exa:concent:iso:approx}}\label{sec:details:of:example}
We start with the following lemma:
\begin{lem}
	Let $\theta$ and $\theta'$ are independent variables, uniformly distributed on the unit sphere $S^{d-1}$. Then, $\sqrt{2d} \ip{\theta,\theta'}  \convd N(0,1)$ as $d \to \infty$.
\end{lem}
\begin{proof}
	Letting $U :=  (\ip{\theta,\theta'} + 1)/2$, one can show that $U \sim \text{Beta}((d-1)/2, (d-1)/2)$. Writing $U = X/(X+Y)$ for independent $\text{Gamma}((d-1)/2,1)$ variables $X$ and $Y$, applying bivariate CLT to $(X,Y)$, followed by the delta method, gives $\sqrt{2 d} (2U - 1) \convd N(0,1)$ as $d \to \infty$ which is the desired result.
\end{proof}
Thus, for $d$ large enough, $\ip{\theta,\theta'}$ is approximately distributed as $N(0,1/(2d))$.
Recalling that
$\psi_d(u) = \ex \exp(u \ip{\theta,\theta'})$ and using the fact that $\ex[e^{\lambda Z}] = \exp(\frac12 \lambda^2 \sigma^2)$ for $Z \sim N(0,\sigma^2)$, we obtain the claimed approximation $\psi_d(u) \approx \exp(u^2 / 4d)$ for $u \ll d$.

\subsection{Details of Example~\ref{exa:concent:aniso:approx}}\label{sec:details:aniso:approx}
Fix $r_k$ and $r_\ell$ and recall that $\theta$ and $\theta'$ are uniformly distributed on $S^{d-1}$. We write $f_d(\alpha) = \Kt_\sigma(r_k \theta,r_\ell\theta')$ where $\Kt_\sigma$ is given by~\eqref{eq:Kt:def:anisotrop} and $\alpha = \ip{\theta,\theta'}$. Note that this definition of $\alpha$ matches that used in~\eqref{eq:Kt:def:anisotrop} with $u=r_k \theta$ and $v=r_\ell \theta'$. Let us define
\begin{align*}
\fb(\alpha) = \exp \Big[ {-}\frac1{2\tau^2} (r_1^2 + r_2^2 - 2 \alpha r_1 r_2) \Big].
\end{align*}
It is not hard to see $f_d(\alpha) \to \fb(\alpha)$ uniformly as $d \to \infty$. In fact, $\sup_\alpha |f_d(\alpha) - \fb(\alpha)| \le C /d $ where the constant $C$ only depends on $\sigma^2/ \tau^2$. It follows that for the mean and variance, we can pass from $f_d(\alpha)$ to $\fb(\alpha)$. The rest of the argument follows as that of Example~\ref{exa:concent:iso:approx}.

\subsection{Mean Gaussian kernel}\label{app:Gauss:ker:comp}
In this appendix, we derive the mean kernel matrix $\ex K$ for the Gaussian kernel~\eqref{eq:Gauss:ker} under the Gaussian data model~\eqref{eq:simp:Gauss:model}. In fact, it is easier to work with the rescaled version of the model: $X_i = \mu_i + \sigma_i z_i$ where $z_i$ are iid $N(0,I_d)$.
Fix $i \neq j$, let $\sigma_{ij}^2 := \sigma_i^2 + \sigma_j^2$ and $m^{ij} = (\mu_i - \mu_j)/\tau \in \reals^d$. Note that $w_{ij} := (\sigma_i z_i - \sigma_j z_j)/\sigma_{ij} \sim N(0,I_d)$.  We have
\begin{align*}
K(X_i,X_j) = \ex \exp\Big({-\frac1{2} \big\| m^{ij} + 
		\frac{\sigma_{ij}}{\tau}w_{ij}} \big\|^2 \Big).
\end{align*}
\begin{lem}\label{lem:Kt:1d}
	Let  $w \sim N(0,1)$. Then, for any $m,t \in \reals$,
	\begin{align*}
		\Kt(m;t) := \ex \exp \Big[{-}\frac12 (m + t w)^2\Big] 
			=  \frac1{s} \exp\Big({-}\frac{m^2}{2 s^2} \Big)
	\end{align*}
	where $s^2 = 1+t^2$. %
\end{lem}
 Applying the lemma, setting $s_{ij}^2 = 1+ (\sigma_{ij}/\tau)^2$, we have
\begin{align*}
\ex K(X_i,X_j) 
&= \prod_{k=1}^d \ex \exp \Big[{-\frac1{2} \Big( m^{ij}_k 
	+ \frac{\sigma_{ij}}{\tau} [w_{ij}]_k \Big)^2 } \Big] \\
&= \prod_{k=1}^d \Kt\Big( m^{ij}_k; \frac{\sigma_{ij}}{\tau} \Big)\\
&= \prod_{k=1}^d \frac1{s_{ij}} \exp\Big( {- \frac{(m^{ij}_k)^2}{2s_{ij}^2}}  \Big)
= \frac1{s_{ij}^d}\exp\Big(  {-\frac{\norm{m^{ij}}^2}{2s_{ij}^2}} \Big), \quad i \neq j
\end{align*}
which is the desired result (after changing $\sigma_j$ to $\sigma_j/\sqrt{d}$).

\begin{proof}[Proof of Lemma~\ref{lem:Kt:1d}]
	We have
	\begin{align*}
	\Kt(m;t) = \frac1{\sqrt{2 \pi}} \int_{\reals} e^{-x^2/2} e^{-(m+tx)^2/2}dx
	\end{align*}
	Letting $s^2 = 1+t^2$, we obtain
	\begin{align*}
	x^2 + (m+tx)^2 &= s^2 x^2 + m^2 + 2mtx \\
	&= s^2(x + mts^{-2})^2 - (mt)^2s^{-2} + m^2 \\
	&= s^2(x+ mts^{-2})^2 + m^2 s^{-2}
	\end{align*}
	using $1 - t^2s^{-2} = s^{-2}$. It follows that
	\begin{align*}
	\Kt(m;t) &= \frac1{\sqrt{2 \pi}} e^{-m^2 s^{-2}/2} \int_{\reals} e^{-s^2(x+mts^{-2})^2/2}dx \\
	&= \frac1s e^{-m^2 s^{-2}/2} \int_{\reals} \frac{s}{\sqrt{2 \pi}}e^{-s^2(x+mts^{-2})^2/2}dx.
	\end{align*}
	The integral is equal to 1 since the integrand is a Gaussian probability density.
\end{proof}

\end{document}